\documentclass[10pt]{amsart}
\usepackage{amssymb}

\usepackage{blindtext,stmaryrd}
\usepackage[utf8]{inputenc}
\newtheorem{thm}{Theorem}[section]
\newtheorem{ques}{Question}

\theoremstyle{definition}

\newtheorem{lem}[thm]{Lemma}
\newtheorem{prop}[thm]{Proposition}
\newtheorem{defn}[thm]{Definition}

\theoremstyle{remark}

\numberwithin{equation}{section}

\begin{document}

\title{On extensions of partial isometries} 

\author{Mahmood Etedadialiabadi}

\address{Department of Mathematics, University of North Texas, 1155 Union Circle \#311430, Denton, TX 76203, USA}

\email{mahmood.etedadialiabadi@unt.edu}

\author{Su Gao}\

\address{Department of Mathematics, University of North Texas, 1155 Union Circle \#311430, Denton, TX 76203, USA}\email{sgao@unt.edu}

\thanks{The second author's research was partially supported by the NSF grant DMS-1800323.}




\begin{abstract}
In this paper we define a notion of S-extension for a metric space and study minimality and coherence of S-extensions. We show that every S-extension can be identified with an algebraic object. We use this algebraic representation to give a complete characterization of all finite minimal S-extensions of a given finite metric space and a complete characterization of all minimal coherent S-extensions. We also define a notion of ultraextensive metric spaces and show that every countable metric space can be extended to a countable ultraextensive metric space. 
We also show that the isometry group of an infinite ultraextensive metric space has a dense locally finite subgroup, generalizing results in \cite{P}\cite{R}\cite{S}\cite{S2}.We also study compact ultrametric spaces and show that every compact ultrametric space can be extended to a compact ultraextensive ultrametric space. 
\end{abstract}

\maketitle
\bigskip\noindent
\section{Introduction}
The study in this paper was motivated by the following theorem of Solecki \cite{S}:
\begin{thm}[Solecki \cite{S}]\label{thm_S}
Every finite metric space $X$ can be extended to a finite metric space $Y$ such that every partial isometry of $X$ extends to an isometry of $Y$.
\end{thm}

The main objective of this paper is to identify such extensions and related concepts with algebraic objects. Inspired by Solecki's theorem, we define the following notions. Given a metric space $X$, a distinguished point $a_0\in X$ and a collection $P$ of partial isometries of $X$ such that $P=P^{-1}$ and $X\setminus \{a_0\}\subseteq P(a_0)$, a {\it $P$-type S-extension} is a pair $(Y,\phi)$, where $Y$ is a metric space extending the metric space $X$ and $\phi$ is a map from $P$ into the set of all isometries of $Y$ such that $\phi(p)$ extends $p$ for all $p\in P$. When $P$ is the set of all partial isometries of $X$, we call $(Y,\phi)$ an {\it S-extension} of $X$. Solecki's theorem can be restated as: Every finite metric space has a finite S-extension. We should also mention that similar notions in a more general context of finite relational structures have been studied by the authors in \cite{EG2}.

If $X$ is a metric space and $(Y, \phi)$ is a $P$-type S-extension of $X$, then we say $(Y,\phi)$ is {\it minimal} (sometimes called {\it irreducible} in the literature) if for all $y\in Y$ there are partial isometries $p_1,\dots, p_n\in P$ and $x\in X$ such that
$$ y=\phi(p_1)\cdots\phi(p_n)(x). $$
This allows us to associate the point $y\in Y$ with $\phi(p_1)\cdots\phi(p_n)\in \phi(\mathbb{F}(P))$, where $\mathbb{F}(P)$ is the free group generated by elements of $P$, and therefore to view $Y$ as an algebraic object. 

Our first main result of the paper is to give an algebraic characterization of all finite minimal $P$-type S-extensions of a given finite metric space. Before even stating the result, we will give a direct constructive proof of Theorem~\ref{thm_S}. Recall that Solecki's proof in \cite{S} uses a result of Herwig--Lascar \cite{HL}, which is in turn a generalization of a celebrated result of Hrushovski \cite{H} on extending partial isomorphisms of finite graphs. Our proof of Theorem~\ref{thm_S} follows the ideas of Herwig--Lascar's proof and is essentially the same as the approaches in Rosendal's in \cite{R} and Pestov's in \cite{P}; but since we have a focus of characterizing all minimal S-extensions, our proof is somewhat different from them. Specifically, our proof is not as general as \cite{R}, and unlike \cite{P}, it does not need the full generality of Herwig--Lascar's result. Here we should mention that Hubi\v{c}ka--Kone\v{c}n\'{y}--Ne\v{s}et\v{r}il \cite{HKN} provided a combinatorial proof of Theorem~\ref{thm_S}.

In a sense, our proof of Theorem~\ref{thm_S} gives a ``canonical'' algebraic construction of a $P$-type S-extension from a parameter we call a {\em feasible prekernel}, which is a suitable normal subgroup of $\mathbb{F}(P)$. Given a feasible prekernel $N\unlhd\ \mathbb{F}(P)$, we construct a canonical algebraic S-extension $(\Gamma_N, \Phi_N)$, and introduce a weight function $w_N$. The minimal S-extensions can then be characterized as follows.

\begin{thm}
Let $(Y,\phi)$ be a finite minimal $P$-type S-extension of $X$. Let $N=\mbox{\rm ker}(\phi)$ and $G=\Phi_N(\mathbb{F}(P))$. Then there is a $G$-invariant pseudometric $\rho$ on $\Gamma_N$ which is consistent with $w_N$ such that $(Y,\phi)$ is isomorphic to $(\overline{\Gamma_N}^{\rho}, \overline{\Phi_N}^{\rho})$. 
\end{thm}

The next notion we study is that of coherence between $P$-type S-extensions. Here we introduce a notion of coherence between feasible prekernels, and our second main result demonstrates a correspondence between the two coherence notions. 

\begin{thm}
Let $X_1\subseteq X_2$ be finite metric spaces, $(Y_1, \phi_1)$ be a minimal $P_1$-type S-extension of $X_1$, and $P_1\subseteq P_2$ where $P_2=P_2^{-1}$ and $X_2\setminus \{a_0\}\subseteq P_2(a_0)$. Then the following hold.
\begin{enumerate}
\item[(i)] Let $(Y_2,\phi_2)$ be a $P_2$-type S-extension of $X_2$ that is coherent with $(Y_1,\phi_1)$. Then $N_2=\mbox{ker}(\phi_2)$ is a coherent extension of $N_1=\mbox{ker}(\phi_1)$.

\item [(ii)] Let $N_2\unlhd\, \mathbb{F}(P_2)$ be a coherent extension of $N_1=\mbox{ker}(\phi_1)$. Then letting $G_2=\Phi_{N_2}(\mathbb{F}(P_2))$, there exists a $G_2$-invariant pseudometric $\rho_2$ on $\Gamma_{N_2}$ which is consistent with $w_{N_2}$,  such that $(Y_2,\phi_2)= (\overline{\Gamma_{N_2}}^{\rho_2}, \overline{\Phi_{N_2}}^{\rho_2})$ is coherent with $(Y_1, \phi_1)$.
\end{enumerate} 
\end{thm}

Iterative coherent S-extensions lead to infinite metric spaces with striking properties. We introduce the notion of ultraextensive metric spaces and obtain some general results about them. Specifically, we call a metric space $U$ {\it ultraextensive} if $U$ is ultrahomogeneous, every finite $X\subseteq U$ has a finite S-extension $(Y, \phi)$ where $Y\subseteq U$, and if $X_1\subseteq X_2\subseteq U$ are finite and $(Y_1,\phi_1)$ is a finite minimal S-extension of $X_1$ with $Y_1\subseteq U$, then there is a finite minimal S-extension $(Y_2, \phi_2)$ of $X_2$ such that $Y_2\subseteq U$ and $(Y_1,\phi_1)$ and $(Y_2,\phi_2)$ are coherent.

Recall ultrahomogeneity means that any partial isometry can be extended to a full isometry of the entire space. Thus ultraextensiveness is a strengthening of ultrahomogeneity. It follows from our results that the universal Urysohn metric space $\mathbb{U}$ and the rational universal Urysohn space $\mathbb{QU}$ are ultraextensive. Another example of ultraextensive metric space is the countable random graph equipped with the path metric. Moreover, we will also establish the following results.

\begin{thm} Every countable metric space can be extended to a countable ultraextensive metric space.
\end{thm}

\begin{thm} If $U$ is an ultraextensive metric space, then every countable subset $X\subseteq U$ can be extended to a countable ultraextensive $Y\subseteq U$.
\end{thm}

\begin{thm} For any separable ultraextensive metric space $U$, $\mbox{Iso}(U)$ contains a dense locally finite subgroup. 
\end{thm}

For compact ultrametric spaces we prove the following.

\begin{thm} Every compact ultrametric space can be extended to a compact ultraextensive ultrametric space. In particular, every compact ultrametric space has a compact ultrametric S-extension.
\end{thm}


The rest of the paper is organized as follows. In Section 2 we give some preliminaries of S-extensions, metrics on weighted graphs, and the profinite topology. In Section 3 we give the canonical algebraic construction of finite S-extensions in the style of Herwig--Lascar. In Section 4 we study finite minimal S-extensions and give a complete characterization of them. In Section 5 we  study the notion of coherent S-extensions and give a complete characterization of coherent S-extensions along with several constructions and applications. In Section 6 we study ultraextensive spaces and establish the main results mentioned above. In Section 7 we study compact ultrametric spaces and show that they admit compact ultrametric S-extensions. In Section 8 we mention some open problems.

\section{Preliminaries\label{Preliminaries}}

\subsection{S-extensions}

We fix some notations to be used in the rest of the paper. Let $(X, d_X)$ and $(Y,d_Y)$ be metric spaces. When there is no danger of confusion, we simply write $X$ for $(X, d_X)$ and $Y$ for $(Y, d_Y)$.

We say that $Y$ is an {\it extension} of $X$ if $X\subseteq Y$ and for all $x_1, x_2\in X$, $d_Y(x_1, x_2)=d_X(x_1, x_2)$. Interchangeably, we use the same terminology when $Y$ contains an isometric copy of $X$.

An {\it isometry} from $X$ to $Y$ is a bijection $\pi: X\to Y$ such that 
$$ d_Y(\pi(x_1),\pi(x_2))=d_X(x_1,x_2) $$
for all $x_1,x_2\in X$. An isometry from $X$ to $X$ is also called an {\it isometry of} $X$. The set of all isometries of $X$ is denoted as $\mbox{Iso}(X)$. Under composition of maps, $\mbox{Iso}(X)$ becomes a group.

A {\it partial isometry} of $X$ is an isometry between two finite subspaces of $X$. The set of all partial isometries of $X$ is denoted as $\mathcal{P}(X)$. $\mathcal{P}(X)$ is not necessarily a group, but it is a groupoid. In particular, for each $p\in \mathcal{P}(X)$ we can speak of the inverse map $p^{-1}$, which is still a partial isometry. 

If $Y$ is an extension of $X$, then every partial isometry of $X$ is also a partial isometry of $Y$. In symbols, we have $\mathcal{P}(X)\subseteq \mathcal{P}(Y)$ if $X\subseteq Y$.

If $p, q\in \mathcal{P}(X)$, we say that $q$ {\it extends} $p$, and write $p\subseteq q$, if 
$$\{ (x, p(x))\,:\, x\in\mbox{dom}(p)\}\subseteq \{ (x, q(x))\,:\, x\in\mbox{dom}(q)\}.$$

We let $1_X$ denote the identity isometry on $X$, i.e., $1_X(x)=x$ for all $x\in X$. Let $\mathcal{P}_X$ denote the set of all $p\in \mathcal{P}(X)$ such that $p\not\subseteq 1_X$. We refer to elements of $\mathcal{P}_X$ as {\it nonidentity partial isometries} of $X$. 

The main concept we study in this paper is that of a $P$-type S-extension.

\begin{defn}
Let $X$ be a finite metric space and fix a distinguished point $a_0\in X$. Let $P\subseteq \mathcal{P}_X$ be such that $P=P^{-1}$ and $X\setminus \{a_0\}\subseteq P(a_0)$. A {\it $P$-type S-extension} of $X$ is a pair $(Y, \phi)$, where $Y\supseteq X$ is an extension of $X$, and $\phi: P\to \mbox{Iso}(Y)$ such that $\phi(p)$ extends $p$ for all $p\in P$. The map $\phi$ is called a {\it $P$-type S-map} for $X$. 

When $P= \mathcal{P}_X$, we call $(Y, \phi)$ an {\it S-extension} of $X$ and $\phi$ an {\it S-map} for $X$.
\end{defn}

Note that an equivalent restatement of Solecki's theorem (Theorem~\ref{thm_S}) is that every finite metric space has a finite S-extension. It is well-known that the universal Urysohn space $\mathbb{U}$ is both universal (for all separable metric spaces) and ultrahomogeneous. These imply that every separable metric space has an S-extension $(Y, \phi)$ where $Y$ is isometric with $\mathbb{U}$. 

We will need the following notion of isomorphism between $P$-type S-extensions.

\begin{defn} Let $X$ be a metric space and $(Y, \phi)$ and $(Z, \psi)$ be both $P$-type S-extensions of $X$. An {\it isomorphism} between $(Y, \phi)$ and $(Z, \psi)$ is an isometry $\pi: Y\to Z$ such that $\psi(p)\circ\pi=\pi\circ\phi(p)$ for all $p\in P$. If there is an isomorphism between $(Y, \phi)$ and $(Z, \psi)$, we say that $(Y, \phi)$ and $(Z,\psi)$ are {\it isomorphic}, and write $(Y, \phi)\cong (Z,\psi)$.
\end{defn}

\subsection{Metrics on weighted graphs}
We will study metric spaces derived from weighted graphs. A {\it weighted graph} is a pair $(\Gamma, w)$, where $\Gamma=(V(\Gamma), E(\Gamma))$ is a (simple undirected) graph and $w: E(\Gamma)\to \mathbb{R}_+$ ($\mathbb{R}_+$ denotes the set of all positive real numbers). We call $w$ the {\it weight function}. If $w_1, w_2$ are two weight functions on $\Gamma$, then we write $w_1\leq w_2$ if $w_1(x,y)\leq w_2(x,y)$ for all $(x,y)\in E(\Gamma)$.

Given a weighted graph $(\Gamma, w)$, let $L_w=\inf\{w(x,y)\,:\, (x, y)\in E(\Gamma)\}$ and $B_w=\sup\{w(x,y)\,:\, (x,y)\in E(\Gamma)\}$. Assuming $0<L_w\leq B_w<\infty$, one can define a {\it path metric} $d_w$ on $V(\Gamma)$ as follows: for any $x,y \in V(\Gamma)$, let
$$
d_w(x,y)=\min\{B_w, \delta_w(x,y)\} $$
where
$$ \delta_w(x,y)=\inf\left\{\sum_{i=1}^n w(x_i,x_{i+1}) \, : \,  x_1=x, \ x_{n+1}=y, \forall i\leq n\ (x_i,x_{i+1})\in E(\Gamma)\right\}.
$$
In particular, $\delta_w(x,y)$ is undefined when $x$ and $y$ are not connected by a path in $\Gamma$, in which case $d_w(x,y)=B_w$. It is easy to verify that $d_w$ is indeed a metric. Also, if $(x,y)\in E(\Gamma)$, then $d_w(x,y)=\delta_w(x,y)$. We note the following simple fact about $d_w$ without proof.

\begin{lem}\label{reduced} The following are equivalent:
\begin{enumerate}
\item[(i)] For any $(x,y)\in E(\Gamma)$, $d_w(x,y)=w(x,y)$.
\item[(ii)] For any $(x,y)\in E(\Gamma)$ and any $x_1, \dots, x_{n+1}$ where $x_1=x$, $x_{n+1}=y$, and $(x_i,x_{i+1})\in E(\Gamma)$ for all $i=1,\dots, n$, we have
$$ w(x,y)\leq \sum^n_{i=1}w(x_i,x_{i+1}). $$
\end{enumerate}
\end{lem}

We introduce some new concepts about the consistency of metrics on weighted graphs.

\begin{defn} Let $(\Gamma, w)$ be a weighted graph and $d$ be a metric on $V(\Gamma)$. 
We say that $d$ is {\it consistent} with $w$ if for all $(x, y)\in E(\Gamma)$, $d(x,y)=w(x,y)$. 
We say that $w$ is {\it reduced} if $d_w$ is consistent with $w$.
\end{defn}

\begin{lem} Let $(\Gamma, w)$ be a connected weighted graph. Then there is a maximal reduced weight function $w^*$ on $\Gamma$ with $w^*\leq w$.
\end{lem}

\begin{proof}
For all $(x,y)\in E(\Gamma)$, define $w^*(x,y)=d_w(x,y)=\delta_w(x,y)$. Then $w^*\leq w$. To see that $w^*$ is reduced we use Lemma~\ref{reduced} and consider $(x,y)\in E(\Gamma)$. Suppose $x_1, \dots, x_{n+1}\in V(\Gamma)$ with $x_1=x$, $x_{n+1}=y$, and $(x_i,x_{i+1})\in E(\Gamma)$ for all $i=1,\dots, n$. Let $\epsilon>0$. For each $i=1,\dots, n$, let $x_i^1=x_i, x_i^2,\dots, x_i^{k_i+1}=x_{i+1}\in V(\Gamma)$ with $(x_i^j, x_i^{j+1})\in E(\Gamma)$ for all $j=1,\dots, k_i$ be such that
$$ w^*(x_i,x_{i+1})=d_w(x_i,x_{i+1})\leq \sum_{j=1}^{k_i}w(x_i^j,x_i^{j+1})\leq w^*(x_i,x_{i+1})+\epsilon/n. $$
Then 
$$w^*(x,y)=d_w(x,y)\leq \sum_{i=1}^n\sum_{j=1}^{k_i} w(x_i^j, x_i^{j+1})\leq \sum_{i=1}^n w^*(x_i,x_{i+1})+\epsilon. $$
Since $\epsilon$ is arbitrary, we have that 
$$ w^*(x,y)\leq \sum_{i=1}^n w^*(x_i,x_{i+1}). $$
Thus $w^*$ is reduced. For the maximality of $w^*$, assume $u\leq w$ is a reduced weighted function. Then for all $(x,y)\in E(\Gamma)$, $u(x,y)=d_u(x,y)\leq d_w(x,y)=w^*(x,y)$.
\end{proof}

We can always turn a metric space into a weighted graph. If $(X, d)$ is a metric space, for any $x, y\in X$ with $x\neq y$, we add an edge between $x$ and $y$ with weight $w_d(x,y)=d(x,y)$. Then $(X, w_d)$ is a connected weighted graph and $w_d$ is a reduced weight function.

We will also consider pseudometrics on weighted graphs. 

\begin{defn} Let $(\Gamma, w)$ be a weighted graph and $\rho$ be a pseudometric on $V(\Gamma)$. We say that $\rho$ is {\it consistent} with $w$ if for all $(x,y)\in E(\Gamma)$, $\rho(x,y)=w(x,y)$.
\end{defn}

When a weight function $w$ satisfies $B_w<\infty$ and $L_w=0$, one can similarly define a {\it path pseudometric} $d_w$ and the distance function $\delta_w$ the same way as above. The resulting path pseudometric is consistent with $w$.

\begin{defn} Let $(M,\rho)$ be a pseudometric space. An {\it isometry} of $(M,\rho)$ is a map $\varphi: M\to M$ such that for all $x,y\in M$, $\rho(\varphi(x), \varphi(y))=\rho(x,y)$. If $G$ is a set of isometries of $(M,\rho)$, we say that $\rho$ is {\it $G$-invariant}.
\end{defn}

Any pseudometric space $(M, \rho)$ has a metric identification defined as follows. Let $\sim$ be an equivalence relation defined on $M$ by $x\sim y$ iff $\rho(x,y)=0$. For each $x\in M$, let $[x]_\sim$ denote the $\sim$-equivalence class of $x$. Then we can define $\overline{M}=\overline{M}^{\rho}=M/\sim$ and a metric $\overline{\rho}$ on $\overline{M}$ by $\overline{\rho}([x]_\sim, [y]_\sim)=\rho(x,y)$ for all $x, y\in M$. $(\overline{M},\overline{\rho})$ is called the {\it metric identification} of $(M, \rho)$. If $\varphi$ is an isometry of $(M, \rho)$, then we can define $\overline{\varphi}: \overline{M}\to\overline{M}$ by $\overline{\varphi}([x]_\sim)=[\varphi(x)]_\sim$ for all $x\in M$. Then $\overline{\varphi}$ is an isometry of $(\overline{M},\overline{\rho})$. Suppose $G$ is a set of isometries of $(M,\rho)$, then $\overline{G}=\overline{G}^{\rho}=\{\overline{\varphi}\,:\, \varphi\in G\}$ is a set of isometries of $(\overline{M},\overline{\rho})$. We note that if $G$ is a group, then $\overline{G}$ is also a group.

Let $(\Gamma,w)$ be a weighted graph and $\rho$ be a pseudometric on $V(\Gamma)$ consistent with $w$. Let $(\overline{\Gamma},\overline{\rho})$ denote the metric identification of the pseudometric space $(\Gamma, \rho)$. For $(x,y)\in E(\Gamma)$, define $\overline{w}([x]_\sim, [y]_\sim)=w(x,y)$. Then $(\overline{\Gamma}, \overline{w})$ is a weighted graph and $\overline{\rho}$ is consistent with $\overline{w}$. We will need this construction in the subsequent sections.

\subsection{The profinite topology} One of the main tools we will be using is Ribes--Zalesskii theorem \cite{RZ} on the profinite topology on an abstract group. Recall that if $G$ is an abstract group, 
the {\it profinite topology} on $G$ is the topology generated by all cosets of normal subgroups of finite index, that is, it has as a basis of open subsets all cosets of normal subgroups of finite index. 

\begin{thm}[Ribes--Zalesskii \cite{RZ}]\label{thmRZ}
Let $\mathbb{F}$ be an abstract free group and $H_1,\dots,H_n$ be finitely generated subgroups of $\mathbb{F}$. Then $H_1\cdots H_n$ is closed in the profinite topology.
\end{thm}

A group $G$ is said to have {\it property RZ} if for any finitely generated subgroups $H_1, \dots, H_n$ of $G$, $H_1\cdots H_n$ is closed in the profinite topology of $G$. All groups with property RZ are residually finite. We will also use the following theorem of Coulbois \cite{C}.
\begin{thm}[Coulbois \cite{C}]\label{Coulbois}
If $G_1$ and $G_2$ have property RZ, then so does the free product $G_1*G_2$.
\end{thm}

Herwig--Lascar \cite{HL} used the Ribes--Zalesskii theorem in their study of the extension problems. This approach was explored further by Rosendal \cite{R} \cite{R2} to study extension problems for isometries, who showed that the Ribes--Zalesskii property for a group $G$ is equivalent to an extension property for actions of $G$ by isometries.

\begin{defn}
Let $G$ be a group acting by isometries on a metric space $(X,d_X)$. We say that the action is {\it finitely approximable} if for any finite $A\subseteq X$ and finite $F\subseteq G$ there is a finite metric space $(Y,d_Y)$, on which $G$ acts by isometries, and an isometry $\pi: A\rightarrow Y$ such that whenever $g\in F$ and $x,gx\in A$, then $\pi(gx)=g\pi (x).$ 
\end{defn}

\begin{thm}[Rosendal \cite{R}]
The following are equivalent for a countable discrete group $G$:
\begin{enumerate}
\item $G$ has property RZ;

\item Any action of $G$ by isometries on a metric space is finitely approximable.
\end{enumerate}
\end{thm}

\section{Finite S-Extensions\label{finite}}

In this section we give a direct constructive proof of Solecki's theorem (Theorem~\ref{thm_S}) following the ideas of Herwig--Lascar \cite{HL}. We will see in the following sections that the construction we present here is in some sense canonical.


For the rest of this section we fix a finite metric space $X$, a distinguished point $a_0\in X$ and $P\subseteq \mathcal{P}_X$ where $P=P^{-1}$ and $X\setminus \{a_0\}\subseteq P(a_0)$. Recall that $\mathcal{P}_X$ is the set of all nonidentity partial isometries of $X$. Let $\mathbb{F}(P)$ be the free group generated by elements of $P$. For each $p\in P$, we identify the partial isometry $p^{-1}\in P$ with the formal inverse of $p$ in $\mathbb{F}(P)$. Thus any nonidentity element of $\mathbb{F}(P)$ is a finite word of the form $p_1\dots p_n$ with $p_1, \dots, p_n\in P$. We use $1$ to denote the identity element of $\mathbb{F}(P)$. Of course, $1$ can be identified with the identity isometry $1_X$. 

Let $H$ be the set of all finite words $p_1\cdots p_n$ with $p_1,\dots, p_n\in P$ such that $p_1\dots p_n(a_0)=p_1(p_2(\cdots p_n(a_0)\cdots))$ is defined and $p_1\dots p_n(a_0)=a_0$. Since $X$ is finite, $H$ is a finitely generated subgroup of 
$\mathbb{F}(P)$. 

Define $\Gamma=\mathbb{F}(P)/H$. We construct a weighted graph $(\Gamma, w)$ as follows: 
\begin{enumerate}
\item for every $p,q\in P\cup\{1\}$ such that $p(a_0)$ and $q(a_0)$ are defined, there is an edge between $pH$ and $qH$ with $w(pH, qH)=d_X(p(a_0),q(a_0))$, and

\item for every $g,g_1,g_2\in \mathbb{F}(P)$, if there is an edge between $g_1H$ and $g_2H$, then there is an edge between $gg_1H$ and $gg_2H$ with $w(gg_1H, gg_2H)=w(g_1H, g_2H)$.
\end{enumerate}
To see that $w$ is well-defined, first note that if $w(g_1H, g_2H)$ is defined then there are $p, q\in P$ with $p(a_0)$ and $q(a_0)$ defined, and $g\in \mathbb{F}(P)$ such that $g_1=gp$ and $g_2=gq$. In this case, $w(g_1H, g_2H)=w(pH, qH)=d_X(p(a_0),q(a_0))$. Thus, to verify that $w$ is well-defined, it suffices to make sure that if $p, q, r, s\in P$ and $d_X(p(a_0),q(a_0))\neq d_X(r(a_0),s(a_0))$, then there does not exist $g\in \mathbb{F}(P)$ such that
$gpH=rH$ and $gqH=sH$. Assume there is such a $g=p_1\dots p_n$. Then $r^{-1}gp, s^{-1}gq\in H$, and thus $r^{-1}p_1\dots p_np(a_0)=a_0$ and $s^{-1}p_1\dots p_nq(a_0)=a_0$. It follows that
$$ p_1\dots p_n(p(a_0))=r(a_0) \mbox{ and } p_1\dots p_n(q(a_0))=s(a_0). $$
Since all $p_1,\dots, p_n$ are partial isometries, we have $d_X(p(a_0), q(a_0))=d_X(r(a_0),s(a_0))$.

From the finiteness of $P$ and the definition of $w$, it is clear that $L_w=\inf\{w(x,y)\,:\, (x,y)\in E(\Gamma)\}>0$ and $B_w=\sup\{w(x,y)\,:\, (x,y)\in E(\Gamma)\}<\infty$. It follows that, equipped with the path metric $d_w$, $\Gamma$ becomes a metric space. When there is no danger of confusion, we use $\Gamma$ to denote the metric space $(\Gamma, d_w)$.

We claim that $\Gamma$ is essentially an extension of $X$. To see this, let $e: X\rightarrow \Gamma$ be defined by 
\[
e(a)=\begin{cases} 
                H, &\text{ if } a=a_0, \\
                pH, &\text{ where }p\in P\text{ and }p(a_0)=a,\text{  if } a\neq a_0. 
          \end{cases} \\
\]
To see that $e$ is well-defined, first note that for any $a\neq a_0$ there is $p\in P$ with $p(a_0)=a$. 
If $p, q\in P$ with $p(a_0)=q(a_0)$, then $p^{-1}q\in H$ and therefore $pH=qH$. Thus $e$ is well-defined. Furthermore, if $e(a)=pH=qH=e(b)$ and $p(a_0)=a$ and $q(a_0)=b$, then $p^{-1}q\in H$ and therefore $a=p(a_0)=q(a_0)=b$. This means that $e$ is one-to-one. To see that $e$ is an isometric embedding, we use the following lemmas.

\begin{lem}\label{lem_weak}
Let $(Y, \phi)$ be a $P$-type S-extension of $X$. Then there is $\pi:\Gamma \rightarrow Y$ such that $d_Y(\pi(g_1H),\pi(g_2H))=w(g_1H,g_2H)$ whenever $(g_1H,g_2H)\in E(\Gamma)$.
\end{lem}
\begin{proof}
We first expand $\phi: P\to \mbox{Iso}(Y)$ to a map $\psi: \mathbb{F}(P)\to \mbox{Iso}(Y)$ by
$$ \psi(p_1\dots p_n)=\phi(p_1)\circ \dots\circ \phi(p_n). $$
Define $\pi:\Gamma \rightarrow Y$ by $\pi(gH)=\psi(g)(a_0)$. It is easy to see that $\pi$ is well-defined. Now, if $(g_1H,g_2H)\in E(\Gamma)$, then there exist $p,q\in P\cup\{1\}$ and $g\in \mathbb{F}(P)$ such that $g_1=gp$, $g_2=gq$, and $p(a_0),q(a_0)$ are defined. We have
\begin{align*}
d_Y(\pi(g_1H),\pi(g_2H))&=d_Y(\pi(gpH),\pi(gqH))=d_Y(\psi(g)\circ\phi(p)(a_0),\psi(g)\circ\phi(q)(a_0))\\ &=d_Y(\phi(p)(a_0),\phi(q)(a_0))=d_X(p(a_0),q(a_0)) \\&=w(pH,qH)=w(gpH, gqH).
\end{align*}
\end{proof}

\begin{lem}\label{reducedw}
$w$ is reduced.
\end{lem}
\begin{proof}
We verify using Lemma~\ref{reduced} that for any $(g_1H, g_2H)\in E(\Gamma)$, $d_w(g_1H, g_2H)=w(g_1H, g_2H)$.  Let $\gamma_1=g_1, \gamma_2,\dots, \gamma_{n+1}=g_2$ be elements of $\mathbb{F}(P)$ such that for all 
$i=1,\dots,n$, $(\gamma_iH,\gamma_{i+1}H)\in E(\Gamma)$. Let $Y=\mathbb{U}$ and $(Y, \phi)$ be a $P$-type S-extension of $X$. Let $\pi: \Gamma\to Y$ be given by Lemma~\ref{lem_weak}. Then 
\begin{align*}
w(g_1H, g_2H)&=d_Y(\pi(g_1H), \pi(g_2H))\\ &\leq \sum_{i=1}^n d_Y(\pi(\gamma_iH),\pi(\gamma_{i+1}H))=\sum_{i=1}^n w(\gamma_iH, \gamma_{i+1}H). 
\end{align*}
\end{proof}

\begin{lem}\label{lem_isometry}
$e$ is an isometric embedding.
\end{lem}
\begin{proof}
Let $a, b\in X$ and $p, q\in P\cup\{1\}$ with $p(a_0)=a$ and $q(a_0)=b$. Then $(pH, qH)\in E(\Gamma)$. Since $w$ is reduced, we have
$$ d_w(e(a), e(b))=d_w(pH, qH)=w(pH, qH)=d_X(p(a_0),q(a_0))=d_X(a, b). $$
\end{proof}

We identify $X$ with 
$$e(X)=\{pH \,:\,  p\in P\cup\{1\} \mbox{ and $p(a_0)$ is defined}\}\subseteq \Gamma$$
and consider $\Gamma$ an extension of $X$. For each $q\in P$, consider the partial map $\hat{q}: e(X)\to e(X)$ defined by
$\hat{q}(pH)=qpH$ for all $pH,qpH\in e(X)$, i.e., whenever $p(a_0)$ and $q(p(a_0))$ are defined; note that in this case there exists $r\in P$ such that $r(a_0)=q(p(a_0))$ and $rH=qpH$. Then it is straightforward to verify that for all $a, b\in X$, $q(a)=b$ iff $\hat{q}(e(a))=e(b)$. Thus we may identify $q$ with $\hat{q}$ on the domain $e(\mbox{dom}(q))$.

Define $\Phi: P\to \mbox{Iso}(\Gamma)$ by letting, for any $q\in P$,
$$ \Phi(q)(gH)=qgH $$
for all $g\in \mathbb{F}(P)$. To see that $\Phi(q)$ is indeed an isometry of $\Gamma$, let $g_1, g_2\in \mathbb{F}(P)$. From the definitions of $w$ and $\delta_w$, we get $\delta_w(g_1H, g_2H)=\delta_w(qg_1H, qg_2H)$ (including the case when one of these quantities is $\infty$). It follows that $d_w(g_1H, g_2H)=d_w(qg_1H, qg_2H)$.

\begin{lem}\label{lem_solution}
$(\Gamma, \Phi)$ is a $P$-type S-extension of $X$.
\end{lem}
\begin{proof}
For any $q\in P$, $\Phi(q)$ is obviously an extension of $\hat{q}$.
\end{proof}

To construct a finite $P$-type S-extension of $X$ our plan is to find a suitable normal subgroup $N$ of finite index in $\mathbb{F}(P)$, and to use $\Gamma_N=\mathbb{F}(P)/NH$ as the underlying space of the $P$-type S-extension. Assuming such a normal subgroup $N\unlhd\, \mathbb{F}(P)$ is found, we first turn $\Gamma_N$ into a weighted graph $(\Gamma_N, w_N)$ as follows:
\begin{enumerate}
\item for every $p,q\in P\cup\{1\}$ with $p(a_0)$ and $q(a_0)$ defined, there is an edge between $pNH$ and $qNH$ with $w_N(pNH, qNH)=d_X(p(a_0),q(a_0))$, and
\item for every $g,g_1,g_2\in \mathbb{F}(P)$, if there is an edge between $g_1NH$ and $g_2NH$, then there is an edge between $gg_1NH$ and $gg_2NH$ with $$w_N(gg_1NH, gg_2NH)=w_N(g_1NH, g_2NH).$$
\end{enumerate}
To guarantee that $w_N$ is well-defined, we use a similar argument as before provided that the following condition holds:
\begin{enumerate}
\item[(C1)] For every $p,q,r,s\in P\cup\{1\}$ such that $d_X(p(a_0),q(a_0))\neq d_X(r(a_0),s(a_0))$, there does not exist $g\in \mathbb{F}(P)$ such that $gpNH=rNH$ and $gqNH=sNH$, equivalently, $N\cap pHr^{-1}sHq^{-1}=\emptyset$.
\end{enumerate}
To see the equivalence in the statement of (C1), suppose $gpNH=rNH$ and $gqNH=sNH$. Then by the normality of $N$ we have $g\in rNHp^{-1}\cap sNHq^{-1}$, and thus $rNHp^{-1}\cap sNHq^{-1}\neq\emptyset$. It follows that
$N\cap pHr^{-1}sHq^{-1}N\neq\emptyset$, or $N\cap pHr^{-1}sHq^{-1}\neq\emptyset$. All steps can be reversed to establish the backward implication.

Another similar argument as before shows that $d_{w_N}$ is a metric on $\Gamma_N$. We again define
\[
e_N(a)=\begin{cases} 
                NH, &\text{ if } a=a_0, \\
                pNH, &\text{ where }p\in P\text{ and }p(a_0)=a,\text{  if } a\neq a_0.
          \end{cases} \\
\]
In order to guarantee that $e_N$ is one-to-one, we argue similarly as before provided that the following condition holds for $N$:
\begin{enumerate}\label{condition}
\item[(C2)] For every $p,q\in P\cup\{1\}$, if $p(a_0)$ and $q(a_0)$ are defined and $p(a_0)\neq q(a_0)$, then $p^{-1}q\notin NH$, equivalently, $N\cap pHq^{-1}=\emptyset$.
\end{enumerate}

Finally, to guarantee that $e_N$ is an isometric embedding, we argue similarly as in the proof of Lemma~\ref{lem_isometry} provided that $w_N$ is reduced, which corresponds to the following condition:
\begin{enumerate}
\item[(C3)] For every $p,q,r_1,s_1,\dots,r_n,s_n\in P\cup\{1\}$ such that 
\[
d_X(p(a_0),q(a_0))>\sum_{i=1}^n d_X(r_i(a_0),s_i(a_0)),
\]
there does not exist a path in $\Gamma_N$ from $pNH$ to $qNH$ using translates of edges ${(r_1NH,s_1NH),\dots, (r_nNH,s_nNH)}$ in the same order. That is, there do not exist $g_1, \dots, g_n\in \mathbb{F}(P)$ such that
\begin{align*} pNH=g_1r_1NH, &\ \ g_1s_1NH=g_2r_2NH, \\  \dots&\dots \\ g_{n-1}s_{n-1}NH=g_nr_nNH, &\ \  g_ns_nNH=qH. 
\end{align*}
Equivalently, $N\cap pHr_1^{-1}s_1H\cdots Hr_n^{-1}s_nHq^{-1}=\emptyset$.
\end{enumerate}

To summarize, we need to find $N\unlhd\, \mathbb{F}(P)$ of finite index so that (C1), (C2) and (C3) hold. Note that these correspond to finitely many conditions, and each condition is of the form $\gamma N\cap H_1\cdots H_n=\emptyset$ where $\gamma\in \mathbb{F}(P)$, $H_1,\dots, H_n$ are finitely generated subgroups of $\mathbb{F}(P)$, and $\gamma\not\in H_1\cdots H_n$. For example, the condition in (C1) can be rewritten as 
$$ (p^{-1}qs^{-1}r)N\cap H (r^{-1}sHs^{-1}r) =\emptyset. $$
Thus, by the Ribes--Zalesskii theorem, for each condition of the form $\gamma N\cap H_1\cdots H_n=\emptyset$, where $\gamma\notin H_1\cdots H_n$, there is a normal subgroup of finite index satisfying the condition. Taking the intersection of all these subgroups, we obtain still a normal subgroup of finite index to satisfy all conditions (C1), (C2) and (C3).

Now 
$$ e_N(X)=\{ pNH\,:\, p\in P\mbox{ and $p(a_0)$ is defined}\}. $$
Similar to the above, for each $q\in P$ we can define the partial map $\tilde{q}: e_N(X)\to e_N(X)$ by $\tilde{q}(pNH)=qpNH$. Then $X$ is identified with $e_N(X)$ and $q$ is identified with $\tilde{q}$ with domain $e(\mbox{dom}(q))$. Define $\Phi_N: P \to \mbox{Iso}(\Gamma_N)$ by 
$$ \Phi_N(q)(gNH)=qgNH. $$
Then it is obvious that $\Phi_N(q)$ extends $\tilde{q}$ for all $q\in P$.

We have thus established that $(\Gamma_N, \Phi_N)$ is a finite $P$-type S-extension of $X$.

\begin{defn}
Let $X$ be a metric space and $P\subseteq \mathcal{P}_X$ be such that $X\setminus \{a_0\}\subseteq P(a_0)$. We say $N\unlhd\, \mathbb{F}(P)$ is a {\it feasible $P$-type prekernel for $X$} if it satisfies (C1), (C2) and (C3). When there is no danger of confusion, we call $N$ a {\it feasible prekernel}.
\end{defn}



What we have shown in this section can be summarized as follows.

\begin{lem} Let $X$ be a metric space. If $N\unlhd\, \mathbb{F}(P)$ is a feasible $P$-type prekernel for $X$, then $(\Gamma_N, \Phi_N)$ is a $P$-type S-extension of $X$.
\end{lem}

Thus Theorem~\ref{thm_S} follows from the fact that for any finite metric space $X$ there is a feasible $P$-type prekernel $N\unlhd\, \mathbb{F}(P)$ that is of finite index in $\mathbb{F}(P)$. This in turn follows from property RZ of the free group $\mathbb{F}(P)$.

\section{Minimal S-Extensions\label{sec_minimal}}

In this section we give a complete characterization of all finite minimal $P$-type S-extensions of a given finite metric space. This is done by showing that the $P$-type S-extension we constructed in the previous section is canonical in several senses. We use the same notations from the previous section.

Throughout this section we still fix a finite metric space $X$, a distinguished point $a_0\in X$ and $P\subseteq \mathcal{P}_X$ where $P=P^{-1}$ and $X\setminus \{a_0\}\subseteq P(a_0)$. We have constructed $P$-type S-extensions $(\Gamma,\Phi)$ and $(\Gamma_N, \Phi_N)$ for suitable $N\unlhd\, \mathbb{F}(P)$. Here we first note that, as long as $P$ is sufficiently rich, these $P$-type S-extensions do not depend on the choice of the point $a_0\in X$. More explicitly, if $a'_0\in X$ and $p_0(a_0)=a'_0$ for $p_0\in P$, and if $X\setminus \{a_0'\}\subseteq P(a_0')$, then we could similarly define $\Gamma'=\mathbb{F}(P)/H'$ and $\Phi'$. It is easy to see that $H'=p_0Hp_0^{-1}$. Thus we may define a bijection $\pi: \Gamma\to \Gamma'$ by $\pi(gH)=gp_0^{-1}H'$ for all $g\in \mathbb{F}(P)$. It is straightforward to check that $\pi$ is an isometry between $\Gamma$ and $\Gamma'$ such that $\pi(\Phi(q)(gH))=\Phi'(q)(\pi(gH))$ for all $q\in P$ and $g\in\mathbb{F}(P)$. Thus $\pi$ is indeed an isomorphism between the two $P$-type S-extensions. Similarly, when $P$ is sufficiently rich, the finite $P$-type S-extension $(\Gamma_N, \Phi_N)$ does not depend on the choice of $a_0$ either.

Next we note that for any $P$-type S-extension $(Y,\phi)$ of $X$, the $P$-type S-map $\phi$ can be trivially extended to a map from all of $\mathbb{F}(P)$ to $\mbox{Iso}(Y)$ by letting
$$ \hat{\phi}(p_1\dots p_n)=\phi(p_1)\circ \dots \circ \phi(p_n) $$
for all $p_1,\dots, p_n\in\mathbb{F}(P)$. $\hat{\phi}$ is a semigroup homomorphism but not necessarily a group homomorphism. To turn it into a group homomorphism, we just need to make sure that $\phi(p^{-1})=\phi(p)^{-1}$ for all $p\in P$, which is easy to arrange. In the rest of this paper, we will use $\phi$ to denote the extension $\hat{\phi}$, and thus regard $\phi$ as a map from $\mathbb{F}(P)$ to $\mbox{Iso}(Y)$. We will also tacitly assume that all the extended $P$-type S-maps $\phi: \mathbb{F}(P)\to \mbox{Iso}(Y)$ are indeed group homomorphisms and therefore their ranges are subgroups of $\mbox{Iso}(Y)$. We note that the extended $P$-type S-map $\Phi: \mathbb{F}(P)\to \mbox{Iso}(\Gamma_N)$ is already a group homomorphism.

The following lemma is one evidence of the canonicity of the construction $(\Gamma_N, \Phi_N)$.

\begin{lem}\label{N_G}
Let $N\unlhd\, \mathbb{F}(P)$ be a feasible $P$-type prekernel for $X$, that is, $(\Gamma_N,\Phi_N)$ is a $P$-type S-extension of $X$. Let $G=\Phi_N(\mathbb{F}(P))\leq \mbox{Iso}(\Gamma_N)$ and $N_G=\mbox{ker}(\Phi_N)$.
Then $N_G\unlhd \, \mathbb{F}(P)$ is a normal subgroup and $NH=N_GH$. In particular, $(\Gamma_{N_G}, \Phi_{N_G})=(\Gamma_N, \Phi_N)$. 
\end{lem}
Note that the notion of ``feasible prekernel'' is justified by Lemma \ref{N_G} which states that such a group $N$ can be massaged into another normal subgroup $N_G$ that produces the same S-extension of $X$ and $N_G=\mbox{ker}(\Phi_N)=\mbox{ker}(\Phi_{N_G})$. 

\begin{proof} $N_G$ is obviously a normal subgroup of $\mathbb{F}(P)$. We only need to verify $NH=N_GH$. Let $\gamma\in N_G$. Then $\Phi_N(\gamma)=1$ and for all $g\in \mathbb{F}(P)$, $\gamma gNH=\Phi_N(\gamma)(gNH)=gNH$. In particular $\gamma NH=NH$, and so $\gamma\in NH$. This shows that $N_G\subseteq NH$ and so $N_GH\leq NH$. Conversely, suppose $\gamma\in N$. Then for all $g\in \mathbb{F}(P)$, we have $\Phi_N(\gamma)(gNH)=\gamma gNH=g(g^{-1}\gamma g)NH=gNH$. Thus $\gamma\in \mbox{ker}(\Phi_N)=N_G$. This shows that $N\leq N_G$ and so $NH\leq N_GH$.
\end{proof}


Next we define minimality for $P$-type S-extensions.

\begin{defn}\label{minimal}
A $P$-type S-extension $(Y,\phi)$ of $X$ is said to be {\it minimal} if for any $y\in Y$ there is $g\in\mathbb{F}(P)$ such that $y=\phi(g)(a_0)$.
\end{defn}

We state the following fact without proof.

\begin{lem} Let $(Y,\phi)$ be a $P$-type S-extension of $X$. Then the following are equivalent:
\begin{enumerate}
\item[(i)] $(Y,\phi)$ is minimal;
\item[(ii)] For any $y\in Y$ there exist $g\in\mathbb{F}(P)$ and $x\in X$ such that $y=\phi(g)(x)$.
\end{enumerate}
\end{lem}

Of course, the notion of minimality is motivated by the observation that if $(Y, \phi)$ is a $P$-type S-extension of $X$ and let
$$ Z=\{ \phi(g)(x)\,:\, g\in\mathbb{F}(P), x\in X\}, $$
then $Z\subseteq Y$ and for any $p\in P$ and $z\in Z$, $\phi(p)(z)\in Z$. Thus, by defining
$$ \psi(p)=\phi(p)\upharpoonright Z $$
for all $p\in P$, we get another $P$-type S-extension $(Z,\psi)$ of $X$ which is a subextension of $(Y, \phi)$.

We also note that, if $(Y,\phi)$ is a $P$-type S-extension of $X$, then there are many ways to define proper superextensions of $(Y,\phi)$ by adding points to $Y$ and defining metrics appropriately. Thus there is no hope to give a reasonable characterization of all finite $P$-type S-extensions of $X$. Below we concentrate on characterizing finite minimal $P$-type S-extensions of $X$. We will show that all finite minimal $P$-type S-extensions of $X$ are derived from $P$-type S-extensions of the form $(\Gamma_N, \Phi_N)$.

\begin{lem}\label{lemc123}
Let $(Y,\phi)$ be a $P$-type S-extension of $X$. Let $N=\mbox{\rm ker}(\phi)$. Then $N$ is a feasible $P$-type prekernel for $X$. 
\end{lem}
\begin{proof}
Define $\Psi: \Gamma_N\to Y$ by $\Psi(g NH)=\phi(g)(a_0)$ for all $g\in\mathbb{F}(P)$. To see $\Psi$ is well-defined, note that if $g_2^{-1}g_1 \in NH$, then for some $n\in N$ and $h\in H$, we have $\phi(g_2)^{-1}\phi(g_1)(a_0)=\phi(n)\phi(h)(a_0)=\phi(n)(a_0)=a_0$. Here we note that for any $n\in N$, $\phi(n)(a_0)=a_0$ since $n\in \mbox{ker}(\phi)$, and for any $h\in H$, $\phi(h)(a_0)=a_0$ since $h(a_0)=a_0$ and $\phi(h)$ extends $h$. Thus $\phi(g_1)(a_0)=\phi(g_2)(a_0)$. 

To verify (C1), let $p,q,r,s\in P\cup \{1\}$ and $g\in \mathbb{F}(P)$ be such that $p(a_0)$, $q(a_0)$, $r(a_0)$ and $s(a_0)$ are defined, $pNH=grNH$ and $qNH=gsNH$. Applying the map $\Psi$ to these equations, we get $\phi(p)(a_0)=\phi(g)\phi(r)(a_0)$ and $\phi(q)(a_0)=\phi(g)\phi(s)(a_0)$. We need to show that $d_X(p(a_0),q(a_0))=d_X(r(a_0),s(a_0))$. Since $\phi(g)$ is an isometry of $Y$, we have 
\begin{align*}
d_X(p(a_0),q(a_0))&=d_Y(\phi(p)(a_0),\phi(q)(a_0)) \\
&=d_Y(\phi(g)\phi(r)(a_0), \phi(g)\phi(s)(a_0)) \\
&=d_Y(\phi(r)(a_0),\phi(s)(a_0))=d_X(r(a_0),s(a_0)).
\end{align*}

To verify (C2), let $p,q\in P\cup \{1\}$ be such that $p(a_0),q(a_0)$ are defined and $pNH=qNH$. We get
$$ p(a_0)=\phi(p)(a_0)=\Psi(pNH)=\Psi(qNH)=\phi(q)(a_0)=q(a_0). $$

Finally, to verify (C3), let $p,q,r_1,s_1,\dots,r_n,s_n\in P\cup\{1\}$ and $g_1, \dots, g_n\in \mathbb{F}(P)$ be such that $p(a_0), q(a_0), r_1(a_0), s_1(a_0), \dots, r_n(a_0), s_n(a_0)$ are all defined, and 
\begin{align*} pNH=g_1r_1NH, &\ \ g_1s_1NH=g_2r_2NH, \\  \dots&\dots \\ g_{n-1}s_{n-1}NH=g_nr_nNH, &\ \  g_ns_nNH=qH. 
\end{align*}
Applying $\Psi$ to all these equations, we get
\begin{align*} p(a_0)=\phi(g_1)(r_1(a_0)), &\ \ \phi(g_1)(s_1(a_0))=\phi(g_2)(r_2(a_0)), \\  \dots&\dots \\ \phi(g_{n-1})(s_{n-1}(a_0))=\phi(g_n)(r_n(a_0)), &\ \  \phi(g_n)(s_n(a_0))=q(a_0). 
\end{align*}
It follows that
\[
d_X(p(a_0),q(a_0))=d_Y(p(a_0), q(a_0))\leq \sum_{i=1}^n d_Y(r_i(a_0),s_i(a_0))=\sum_{i=1}^n d_X(r_i(a_0),s_i(a_0)).
\]
\end{proof}

Thus, for any finite $P$-type S-extension $(Y,\phi)$, we have that $N=\mbox{ker}(\phi)$ is a normal subgroup of $\mathbb{F}(P)$ of finite index and that $N$ is a feasible prekernel. Furthermore, we will be able to carry out the construction of $(\Gamma_N, \Phi_N)$ as in Section~\ref{finite} as a $P$-type S-extension of $X$ based on the weighted graph $(\Gamma_N, w_N)$. In particular, the weight function $w_N$, the path metric $d_{w_N}$, the isometric embedding $e_N$, etc. are all well-defined. 

\begin{thm}\label{lem_onto}
Let $(Y,\phi)$ be a finite minimal $P$-type S-extension of $X$. Let $N=\mbox{\rm ker}(\phi)$ and $G=\Phi_N(\mathbb{F}(P))$. Then there is a $G$-invariant pseudometric $\rho$ on $\Gamma_N$ which is consistent with $w_N$ such that $(Y,\phi)$ is isomorphic to $(\overline{\Gamma_N}, \overline{\Phi_N})$. 
\end{thm}

\begin{proof}
We again define $\Psi: \Gamma_N\to Y$ by $\Psi(g NH)=\phi(g)(a_0)$ for all $g\in\mathbb{F}(P)$. As in the proof of Lemma~\ref{lemc123}, $\Psi$ is well-defined. Since $\phi$ is minimal, $\Psi$ is onto. 

We define a pseudometric $\rho$ on $\Gamma_N$ by 
$$ \rho(g_1NH, g_2NH)=d_Y(\Psi(g_1NH), \Psi(g_2NH))=d_Y(\phi(g_1)(a_0), \phi(g_2)(a_0)). $$
It is easy to verify that $\rho$ is indeed a pseudometric on $\Gamma_N$. 

Recall that $\Phi_N: P\to \mbox{Iso}(\Gamma_N)$ is defined by 
$$ \Phi_N(p)(gNH)=pgNH $$
for all $p\in P$ and $g\in \mathbb{F}(P)$. From previous section, the $\mbox{Iso}(\Gamma_N)$ refers to the group of isometries for the metric space $(\Gamma_N, d_{w_N})$. Here we claim that the maps $gNH\mapsto pgNH$ are also isometries of the pseudometric space $(\Gamma_N, \rho)$. To see this, we only need to check
\begin{align*}
\rho(pg_1NH, pg_2NH)&=d_Y(\phi(pg_1)(a_0), \phi(pg_2)(a_0))\\&=d_Y(\phi(p)\phi(g_1)(a_0),\phi(p)\phi(g_2)(a_0))=\rho(g_1NH, g_2NH).
\end{align*}
Extending $\Phi_N$ to a group homomorphism from $\mathbb{F}(P)$ to $\mbox{Iso}(\Gamma_N)$, the group of all isometries of the pseudometric space $(\Gamma_N, \rho)$, it follows that $\rho$ is $G$-invariant. 

To verify that $\rho$ is consistent with $w_N$, we consider an edge in the weighted graph $(\Gamma_N, w_N)$, which is of the form $(gpNH, gqNH)$ where $g\in \mathbb{F}(P)$ and $p,q\in P\cup\{1\}$ are such that $p(a_0)$ and $q(a_0)$ are defined. Note that $w_N(gpNH, gqNH)=d_X(p(a_0), q(a_0))$. We have
\begin{align*}
\rho(gpNH, gqNH)&=d_Y(\phi(g)\phi(p)(a_0),\phi(g)\phi(q)(a_0))\\&=d_Y(\phi(p)(a_0),\phi(q)(a_0))\\&=d_X(p(a_0),q(a_0))=w_N(gpNH, gqNH).
\end{align*}

We can now consider the metric identification of the pseudometric space $(\Gamma_N, \rho)$, which is denoted by $(\overline{\Gamma_N}, \overline{\rho})$. Since $\rho$ is consistent with $w_N$, so is $\overline{\rho}$. Since $\rho$ is $G$-invariant, for each $\varphi\in G$ we can define an isometry $\overline{\varphi}\in \overline{G}$ for $(\overline{\Gamma_N}, \overline{\rho})$. Thus it makes sense to define $\overline{\Phi_N}: P\to \mbox{Iso}(\overline{\Gamma_N})$ by $\overline{\Phi_N}(p)=\overline{\Phi_N(p)}$. 

Finally, let $\pi: \overline{\Gamma_N}\to Y$ be defined as $\pi([gNH]_\sim)=\phi(g)(a_0)$. Then $\pi$ is an isometry between the metric spaces $(\overline{\Gamma_N}, \overline{\rho})$ and $(Y, d_Y)$. To complete the proof of the theorem, we only need to verify that for any $p\in P$, $\pi\circ\overline{\Phi_N}(p)=\phi(p)\circ\pi$. We have
\begin{align*} [\pi\circ\overline{\Phi_N}(p)]([gNH]_\sim)&=\pi[\overline{\Phi_N}(p)([gNH]_\sim)]=\pi([pgNH]_\sim)\\
&=\phi(pg)(a_0)=\phi(p)\phi(g)(a_0)\\
&=\phi(p)[\pi([gNH]_\sim)]=[\phi(p)\circ\pi]([gNH]_\sim). 
\end{align*}
\end{proof}

\begin{thm} Let $N\unlhd\,\mathbb{F}(P)$ be a feasible $P$-type prekernel that is of finite index. Let $G=\Phi_N(\mathbb{F}(P))$. Let $\rho$ be a $G$-invariant pseudometric on $\Gamma_N$ which is consistent with the weight function $w_N$. Then $(\overline{\Gamma_N},\overline{\Phi_N})$ is a finite minimal $P$-type S-extension of $X$.
\end{thm}

\begin{proof}
Consider the metric identification $(\overline{\Gamma_N},\overline{\rho})$. Since $\rho$ is $G$-invariant, $\overline{G}$ is a set of isometries for $\overline{\Gamma_N}$. Since $\rho$ is consistent with $w_N$, so is $\overline{\rho}$. 
Define $\overline{e_N}: X\to \overline{\Gamma_N}$ by $\overline{e_N}(a)=[e_N(a)]_\sim$. Then $\overline{e_N}$ is an isometric embedding from $X$ into $\overline{\Gamma_N}$. Note that $\overline{e_N}(a_0)=[NH]_\sim$. Thus we can identify $a_0$ with
$[NH]_\sim$. It follows from similar arguments as before that $(\overline{\Gamma_N}, \overline{\Phi_N})$ is a $P$-type S-extension of $X$. To see that it is minimal, we just note that for any $g\in\mathbb{F}(P)$, 
$\overline{\Phi_N}(g)([NH]_\sim)=[gNH]_\sim$. 
\end{proof}

We summarize the characterization of all finite minimal $P$-type S-extensions of $X$ in the following theorem.

\begin{thm} The following are equivalent:
\begin{enumerate}
\item[(i)] $(Y, \phi)$ is a finite minimal $P$-type S-extension of $X$;
\item[(ii)] There exists a feasible $P$-type prekernel $N\unlhd\, \mathbb{F}(P)$ of finite index, and, letting $G=\Phi_N(\mathbb{F}(P))$, there exists a $G$-invariant pseudometric $\rho$ on $\Gamma_N$ which is consistent with $w_N$, such that $(Y,\phi)$ is isomorphic to $(\overline{\Gamma_N}, \overline{\Phi_N})$;
\item[(iii)] For $N=\mbox{\rm ker}(\phi)$ and $G=\Phi_N(\mathbb{F}(P))$, there exists a $G$-invariant pseudometric $\rho$ on $\Gamma_N$ which is consistent with $w_N$, such that $(Y,\phi)$ is isomorphic to $(\overline{\Gamma_N}, \overline{\Phi_N})$.
\end{enumerate}
\end{thm}

\section{Coherent S-Extensions}\label{Coherent S-Extensions}

\subsection{Coherent S-extensions and strongly coherent S-extensions}
In this section we study a notion of coherence for ($P$-type) S-extensions. The terminology has been used for a different notion in Siniora--Solecki \cite{S2} which refers to a slightly stronger condition. We call their notion of coherence \emph{strongly coherent}.

\begin{defn}[Solecki] Let $X$ be a metric space. An S-extension $(Y,\phi)$ of $X$ is \emph{strongly coherent} if for every triple $(p,q,r)$ of partial isometries of $X$ such that $p\circ q=r$, we have $\phi(p)\circ \phi(q)=\phi(r)$. 
\end{defn}

The following is a strengthening of Solecki's theorem on existence of finite S-extensions.

\begin{thm}[Solecki \cite{S09} \cite{R} \cite{S2}]\label{SoleckiTheorem}
Let $X$ be a finite metric space. Then $X$ has a finite strongly coherent S-extension $(Y, \phi)$.
\end{thm}

The following notion of coherence is slightly weaker than the notion of strongly coherent but is sufficient for our study of ultraextensive metric spaces in subsequent sections.
\begin{defn}\label{coherence}
Let $X_1\subseteq X_2$ be metric spaces and $(Y_i,\phi_i)$ be a $P_i$-type S-extension of $X_i$ for $i=1,2$ where $P_1\subseteq P_2$. We say that $(Y_1,\phi_1)$ and $(Y_2, \phi_2)$ are {\it coherent} if 
\begin{enumerate}
\item[(i)] $Y_2$ extends $Y_1$, 
\item[(ii)] $\phi_2(p)$ extends $\phi_1(p)$ for all $p\in P_1\subseteq P_2$, and
\item[(iii)] letting $K_i=\phi_i(\mathbb{F}(P_i))\leq \mbox{Iso}(Y_i)$ for $i=1,2$, and letting $\kappa:K_1\to K_2$ be such that $\kappa(\phi_1(p))=\phi_2(p)$ for all $p\in P_1$, then $\kappa$ has a unique extension to a group isomorphic embedding from $K_1$ into $K_2$.
\end{enumerate}
\end{defn}

The following lemma makes it precise that the notion of strong coherence is a stronger notion than coherence.

\begin{lem}\label{strong} Let $X_1\subseteq X_2$ be finite metric spaces and $(Y_1, \phi_1)$ be a $P_1$-type S-extension of $X_1$. Let $P_2\supseteq P_1$ be such that $P_2=P_2^{-1}$ and $X_2\setminus\{a_0\}\subseteq P_2(a_0)$. Suppose $(Y_2, \phi)$ is a strongly coherent S-extension of $X_2\cup Y_1$. Then there is $\phi_2$ such that $(Y_2, \phi_2)$ is a $P_2$-type S-extension of $X_2$ which is coherent with $(Y_1, \phi_1)$.
\end{lem}

\begin{proof} Let $(Y_2, \phi)$ be a strongly coherent S-extension of $X_2\cup Y_1$. Define $\phi_2: P_2\to \mbox{Iso}(Y_2)$ by
$$ \phi_2(p)=\left\{\begin{array}{ll} \phi(\phi_1(p)), &\mbox{ if $p\in P_1$,}\\ \phi(p), &\mbox{ if $p\in P_2\setminus P_1$.}
\end{array}\right. $$
Then $(Y_2, \phi_2)$ is obviously a $P_2$-type S-extension of $X_2$. The construction also guarantees that $Y_1\subseteq Y_2$ and that $\phi_1(p)\subseteq \phi_2(p)$ for $p\in P_1$. Since $(Y_2, \phi)$ is a strongly coherent S-extension of $Y_1$, $\phi$ restricted to $\mbox{Iso}(Y_1)$ is a group isomorphism embedding from $\mbox{Iso}(Y_1)$ into $\mbox{Iso}(Y_2)$. When further restricted to $K_1=\phi_1(\mathbb{F}(P_1))$, it gives a group isomorphic embedding into $K_2=\phi_2(\mathbb{F}(P_2))$.
\end{proof}

\subsection{A characterization of coherent S-extensions}

Although the existence of coherent S-extensions follows from results of Solecki \cite{S09} \cite{R} \cite{S2} by Lemma~\ref{strong}, we introduce a notion of coherent extensions for groups and utilize this notion to give a characterization of all possible minimal coherent S-extensions. 

Let $X_1\subseteq X_2$ be finite metric spaces, $(Y_1, \phi_1)$ be a minimal $P_1$-type S-extension of $X_1$, and $P_1\subseteq P_2\subseteq \mathcal{P}_{X_2}$ where $P_2=P_2^{-1}$ and $X_2\setminus \{a_0\}\subseteq P_2(a_0)$. Let $(Y_2, \phi_2)$ be a minimal $P_2$-type S-extension of $X_2$ that is coherent with $(Y_1, \phi_1)$. Next, we characterize all such coherent S-extensions.

\begin{defn}
Let $X_1\subseteq X_2$ be finite metric spaces, $(Y_1, \phi_1)$ be a $P_1$-type S-extension of $X_1$ and $P_1\subseteq P_2\subseteq \mathcal{P}_{X_2}$ where $P_2=P_2^{-1}$ and $X_2\setminus \{a_0\}\subseteq P_2(a_0)$. Let $N_1=\mbox{ker}(\phi_1)$. We say $N_2\unlhd\, \mathbb{F}(P_2)$ is a {\it coherent extension of $N_1$} if it is a feasible $P_2$-type prekernel for $X_2$ and satisfies the following conditions:
\begin{enumerate}
\item[(D1)] $N_1=N_2\,\cap\, \mathbb{F}(P_1)$;
\item[(D2)] For every $g, h, k, l\in\mathbb{F}(P_1)$ such that $$d_{Y_1}(\phi_1(g)(a_0), \phi_1(h)(a_0))\neq d_{Y_1}(\phi_1(k)(a_0), \phi_1(l)(a_0)),$$ we have $N_2\cap gH_2k^{-1}l H_2h^{-1}=\emptyset$;
\item[(D3)] For every $g, h\in \mathbb{F}(P_1)$ and $p,q\in P_2$ with both $p(a_0)$ and $q(a_0)$ defined, if $$d_{Y_1}(\phi_1(g)(a_0),\phi_1(h)(a_0))\neq d_{X_2}(p(a_0),q(a_0)),$$
we have $N_2\cap gH_2p^{-1}qH_2h^{-1}=\emptyset$.
\end{enumerate}

\end{defn}

Note that since $N_2$ is a feasible prekernel, letting $G_2=\Phi_{N_2}(\mathbb{F}(P_2))$, for any $G_2$-invariant pseudometric $\rho_2$ on $\Gamma_{N_2}$ which is consistent with $w_{N_2}$, $(\overline{\Gamma_{N_2}}, \overline{\Phi_{N_2}})$ is a minimal S-extension of $X_2$.

\begin{thm}\label{thmcoherent}
Let $X_1\subseteq X_2$ be finite metric spaces, $(Y_1, \phi_1)$ be a minimal $P_1$-type S-extension of $X_1$, and $P_1\subseteq P_2\subseteq \mathcal{P}_{X_2}$ where $P_2=P_2^{-1}$ and $X_2\setminus \{a_0\}\subseteq P_2(a_0)$.
\begin{enumerate}
\item[(i)] Let $(Y_2,\phi_2)$ be a $P_2$-type S-extension of $X_2$ that is coherent with $(Y_1,\phi_1)$. Then $N_2=\mbox{ker}(\phi_2)$ is a coherent extension of $N_1=\mbox{ker}(\phi_1)$.

\item [(ii)] Let $N_2\unlhd\, \mathbb{F}(P_2)$ be a coherent extension of $N_1=\mbox{ker}(\phi_1)$. Then letting $G_2=\Phi_{N_2}(\mathbb{F}(P_2))$, there exists a $G_2$-invariant pseudometric $\rho_2$ on $\Gamma_{N_2}$ which is consistent with $w_{N_2}$,  such that $(Y_2,\phi_2)=(\overline{\Gamma_{N_2}}, \overline{\Phi_{N_2}})$ is coherent with $(Y_1, \phi_1)$.
\end{enumerate} 
\end{thm}

\begin{proof} We first prove (i). Let $(Y_2, \phi_2)$ be a $P_2$-type S-extension of $X_2$ that is coherent with $(Y_1,\phi_1)$. Let $N_2=\mbox{ker}(\phi_2)$. Then by Theorem \ref{lem_onto}, there is a pseudometric $\rho_2$ on $\Gamma_{N_2}$ such that $(Y_2,\phi_2)$ is isomorphic to $(\overline{\Gamma_{N_2}}, \overline{\Phi_{N_2}})$. By Lemma~\ref{lemc123}, $N_2$ is a feasible prekernel. Next we show that $N_2$ is a coherent extension of $N_1$. 

For (D1), note that since $(Y_2, \phi_2)$ is coherent with $(Y_1,\phi_1)$, we have $\phi_1(\mathbb{F}(P_1))\cong \phi_2(\mathbb{F}(P_1))$ via the map $\phi_1(g)\mapsto \phi_2(g)$. Thus, 
\[
N_2\,\cap\, \mathbb{F}(P_1)=\mbox{ker}(\phi_2)\,\cap\, \mathbb{F}(P_1)=\mbox{ker}(\phi_1)=N_1.
\]

For (D2), we need to verify that if for $g, h, k, l\in\mathbb{F}(P_1)$
$$d_{Y_1}(\phi_1(g)(a_0), \phi_1(h)(a_0))\neq d_{Y_1}(\phi_1(k)(a_0), \phi_1(l)(a_0)),$$ 
then $N_2\cap gH_2k^{-1}l H_2h^{-1}=\emptyset$. Toward a contradiction, assume
there is $n\in N_2\cap gH_2k^{-1}l H_2h^{-1}$. Then there are $\eta, \eta'\in H_2$ with $k\eta^{-1}g^{-1}n=l \eta' h^{-1}$, which implies 
$$ \phi_2(k\eta^{-1}g^{-1}n)=\phi_2(l \eta' h^{-1}).$$ 
From the definitions of $H_2$ and of $N_2$, if we apply the left-hand-side element to $\phi_2(g)(a_0)=\phi_1(g)(a_0)$, the resulting value is $\phi_2(k)(a_0)=\phi_1(k)(a_0)$. Similarly, if we apply the right-hand-side element to $\phi_2(h)(a_0)=\phi_1(h)(a_0)$, the resulting value is $\phi_2(l)(a_0)=\phi_1(l)(a_0)$. Thus, both sides of the equation represent the same partial isometry of $Y_1$ with $\phi_1(g)(a_0)$ and $\phi_1(h)(a_0)$ in its domain and with $\phi_1(k)(a_0)$ and $\phi_1(l)(a_0)$ in its range. We conclude that $d_{Y_1}(\phi_1(g)(a_0), \phi_1(h)(a_0))=d_{Y_1}(\phi_1(k)(a_0), \phi_1(l)(a_0))$, a contradiction.

The argument for (D3) is similar. This finishes the proof of (i).
      
For (ii), let $G_1=\Phi_{N_1}(\mathbb{F}(P_1))$ and, by Theorem~\ref{lem_onto}, let $\rho_1$ be a $G_1$-invariant pseudometric that is consistent with $w_{N_1}$ such that $(\overline{\Gamma_{N_1}}, \overline{\Phi_{N_1}})\cong (Y_1,\phi_1)$. For notational simplicity we assume $(Y_1, \phi_1)=(\overline{\Gamma_{N_1}}, \overline{\Phi_{N_1}})$. Let $N_2\unlhd\, \mathbb{F}(P_2)$ be a coherent extension of $N_1=\mbox{ker}(\phi_1)$. Since $N_2$ is a feasible prekernel, one can define $\Gamma_{N_2}$, $w_{N_2}$, and $\Phi_{N_2}$ as before. 

Define a map $\pi:\Gamma_{N_1}\rightarrow \Gamma_{N_2}$ by letting $\pi(gN_1H_1)=gN_2H_2$ for all $g\in\mathbb{F}(P_1)$. To see $\pi$ is well-defined, note that if $gN_1H_1=g'N_1H_1$, then $g^{-1}g'\in N_1H_1\leq N_2H_2$, and therefore  $gN_2H_2=g'N_2H_2$. 

Recall that $w_{N_1}$ is defined for pairs $(gpN_1H_1, gqN_1H_1)$ where $g\in\mathbb{F}(P_1)$ and $p, q\in P_1$ with $p(a_0)$ and $q(a_0)$ defined, and its value is $d_{X_1}(p(a_0), q(a_0))$. Let $\pi(w_{N_1})$ on $\pi(\Gamma_{N_1})$ be the push-forward weight function, that is, $$\pi(w_{N_1})(gpN_2H_2, gqN_2H_2)=w_{N_1}(gpN_1H_1, gqN_1H_1).$$ Note that (D2) implies that $\pi(w_{N_1})$ is well-defined. Also note that $w_{N_2}$ coincides with $\pi(w_{N_1})$ on $\pi(\Gamma_{N_1})$. In fact, $w_{N_2}$ is defined in the same way on the image of such pairs under $\pi$, that is, on pairs of the form $(gpN_2H_2, gqN_2H_2)$ for $g\in \mathbb{F}(P_1)\subseteq \mathbb{F}(P_2)$ and $p, q\in P_1\subseteq P_2$.

Recall that $\Phi_{N_2}$ is defined by $\Phi_{N_2}(p)(gN_2H_2)=pgN_2H_2$ for all $g\in \mathbb{F}(P_2)$ and $p\in P_2$, and is extended to a group homomorphism from $\mathbb{F}(P_2)$ to the symmetric group of $\Gamma_{N_2}$. Let $G_2=\Phi_{N_2}(\mathbb{F}(P_2))$.

We are now ready to define a $G_2$-invariant pseudometric $\rho_2$ on $\Gamma_{N_2}$ that is consistent with $w_{N_2}$ and satisfies $\rho_2\upharpoonright \pi(\Gamma_{N_1})=\pi(\rho_1)$. Here $\pi(\rho_1)$ is the pseudometric on $\pi(\Gamma_{N_1})$ defined by $\pi(\rho_1)(gN_2H_2,hN_2H_2)=\rho_1(gN_1H_1,hN_1H_1)$ for $g, h\in \mathbb{F}(P_1)$. Note that (D2) implies that $\pi(\rho_1)$ is well-defined. 

We define $\rho_2$ as follows. First, for $g, h\in\mathbb{F}(P_1)$, define
\begin{align*}
\rho_2(gN_2H_2, hN_2H_2)&=\pi(\rho_1)(gN_2H_2, hN_2H_2)\\ &=\rho_1(gN_1H_1, hN_1H_1)=d_{Y_1}(\phi_1(g)(a_0), \phi_1(h)(a_0)).
\end{align*}
Next, for $p, q\in P_2$ with $p(a_0)$ and $q(a_0)$ defined, and for $\gamma\in \mathbb{F}(P_2)$, define
$$ \rho_2(\gamma pN_2H_2, \gamma qN_2H_2)=w_{N_2}(\gamma pN_2H_2, \gamma qH_2N_2)=d_{X_2}(p(a_0), q(a_0)). $$
To see that these do not conflict with each other, note that (D3) implies that for $g,h\in\mathbb{F}(P_1)$ and $p, q\in P_2$ with both $p(a_0)$ and $q(a_0)$ defined, if 
$$ d_{Y_1}(\phi_1(g)(a_0),\phi(h)(a_0))\neq d_{X_2}(p(a_0), q(a_0)), $$
then there is no $\gamma\in \mathbb{F}(P_2)$ with $\gamma pN_2H_2=gN_2H_2$ and $\gamma qN_2H_2=hN_2H_2$. We continue to define $\rho_2$ so that if $g, h\in\mathbb{F}(P_2)$ and $\rho_2(gN_2H_2, hN_2H_2)$ is already defined, then we define 
$$ \rho_2(\gamma gN_2H_2, \gamma hN_2H_2)=\rho_2(gN_2H_2, hN_2H_2) $$
for any $\gamma\in\mathbb{F}(P_2)$. To see that this does not create a conflict among the existing definitions of $\rho_2$ values, note that condition (D2) implies that for any $g, h, k, l\in\mathbb{F}(P_1)$, if there is $\gamma\in\mathbb{F}(P_2)$ such that $\gamma gN_2H_2=kN_2H_2$ and $\gamma hN_2H_2=lN_2H_2$, then 
$$ d_{Y_1}(\phi_1(g)(a_0), \phi_1(h)(a_0))=d_{Y_1}(\phi_1(k)(a_0), \phi_1(l)(a_0)). $$
To complete the definition of $\rho_2$, we consider the existing values of $\rho_2$ as a weight function and define $\rho_2$ to be the path pseudometric. Since the weight function is $G_2$-invariant, it follows from the definition of the path pseudometric that the resulting $\rho_2$ is also $G_2$-invariant. 

Since for every $g,h\in \mathbb{F}(P_1)$ we have $$\rho_2(gN_2H_2,hN_2H_2)=\pi(\rho_1)(gN_2H_2,hN_2H_2)=\rho_1(gN_1H_1,hN_1H_1),$$ we also have that $\overline{\rho_2}\upharpoonright \overline{\pi(\Gamma_{N_1})}=\overline{\pi(\rho_1)}$. Note that (D2) with $l,k=1$ implies that the induced map $\overline{\pi}:\overline{\Gamma_{N_1}}=\overline{\Gamma_{N_1}}^{\rho_1}\rightarrow \overline{\Gamma_{N_2}}=\overline{\Gamma_{N_2}}^{\rho_2}$ is an isometric embedding. Thus $\overline{\Gamma_{N_1}}\cong \overline{\pi(\Gamma_{N_1})}$ is a subspace of $\overline{\Gamma_{N_2}}$.
 
Letting $Y_2=\overline{\Gamma_{N_2}}$ and $\phi_2=\overline{\Phi_{N_2}}$. We have that $(Y_2, \phi_2)$ is a $P_2$-type S-extension of $X_2$ and $Y_1\subseteq Y_2$ via the isomorphism of $(Y_1,\phi_1)$ with $(\overline{\Gamma_{N_1}}, \overline{\Phi_{N_1}})$. To see the coherence of $(Y_2, \phi_2)$ with $(Y_1,\phi_1)$, let $p\in P_1$. Then 
$\Phi_{N_1}(p)(gN_1H_1)=pgN_1H_1$ for all $g\in\mathbb{F}(P_1)$ and 
$\Phi_{N_2}(p)(gN_2H_2)=pgN_2H_2$ for all $g\in\mathbb{F}(P_2)$. Via the induced embedding $\overline{\pi}:\overline{\Gamma_{N_1}}\rightarrow \overline{\Gamma_{N_2}}$ and the isomorphism of $(Y_1,\phi_1)$ with $(\overline{\Gamma_{N_1}}, \overline{\Phi_{N_1}})$, and because $\rho_2\upharpoonright \pi(\Gamma_{N_1})=\pi(\rho_1)$, we have $\phi_1(p)\subseteq \phi_2(p)$. Finally, it is clear that the map $\phi_1(p)\mapsto \phi_2(p)$ for all $p\in P_1$ generates a group isomorphic embedding from $G_1$ to $G_2$.
\end{proof}

\subsection{A construction of coherent prekernel extensions}
The existence of finite coherent S-extensions via Lemma~\ref{strong} and Theorem \ref{thmcoherent} imply the existence of coherent prekernel extensions. In this subsection we provide a direct construction of coherent prekernel extensions of finite index. 

\begin{lem}\label{Ucoherent}
Let $X_1\subseteq X_2$ be finite metric spaces, $(Y_1, \phi_1)$ be a $P_1$-type S-extension of $X_1$ and $P_1\subseteq P_2\subseteq \mathcal{P}_{X_2}$ where $P_2=P_2^{-1}$ and $X_2\setminus \{a_0\}\subseteq P_2(a_0)$. Then there exists a $P_2$-type S-map $\phi_{\mathbb{U}}: \mathbb{F}(P_2)\to \mbox{\rm Iso}(\mathbb{U})$ such that $(\mathbb{U},\phi_{\mathbb{U}})$ is a $P_2$-type S-extension of $X_2$ which is coherent with $(Y_1,\phi_1)$.
\end{lem}

\begin{proof} Following Uspenskij's proof in \cite{U}, which uses the Katetov construction of $\mathbb{U}$ to show that the isometry group of every Polish space can be embedded into $\mbox{Iso}(\mathbb{U})$ (see also Sections 1.2 and 2.5 of \cite{G} for details), we obtain an isometric embedding $i: Y_1\to \mathbb{U}$ and a group isomorphic embedding $j:\mbox{Iso}(Y_1)\to\mbox{Iso}(\mathbb{U})$ such that for every $\varphi\in \mbox{Iso}(Y_1)$, $j(\varphi)\supseteq \varphi$. In addition, from the ultrahomogeneity of $\mathbb{U}$ we obtain an isometric copy of $X_2$ in $\mathbb{U}$ as a superset of $X_1$. Now for each $p\in P_2$, let $\phi_{\mathbb{U}}(p)\in \mbox{Iso}(\mathbb{U})$ be an extension of $p$ guaranteed to exist by the ultrahomogeneity of $\mathbb{U}$ such that if $p\in P_1$, then $\phi_{\mathbb{U}}(p)=j(\phi_1(p))$. Then $\phi_{\mathbb{U}}$ is as required.  
\end{proof}

\begin{prop}\label{finiteindexcoherent}
Suppose $X_1\subseteq X_2$ are finite metric spaces and $P_1\subseteq P_2\subseteq \mathcal{P}_{X_2}$ where $P_2=P_2^{-1}$ and $X_2\setminus \{a_0\}\subseteq P_2(a_0)$. Let $(Y_1,\phi_1)$ be a finite $P_1$-type S-extension of $X_1$ and $N_1=\mbox{ker}(\phi_1)$. Then, there exists a coherent extension, $N_2\unlhd\, \mathbb{F}(P_2)$, of $N_1$ of finite index.
\end{prop}

\begin{proof}
By Lemma~\ref{lemc123}, $N_1=\mbox{ker}(\phi_1)$ is a feasible prekernel. We may define $w_{N_1}$ and $\Phi_{N_1}$ and let $G_1=\Phi_{N_1}(\mathbb{F}(P_1))$ and $\Gamma_{N_1}=\mathbb{F}(P_1)/N_1H_1$. Since $(Y_1,\phi_1)$ is a minimal $P_1$-type S-extension of $X_1$, by Theorem \ref{lem_onto}, there is a $G_1$-invariant pseudometric $\rho_1$ on $\Gamma_{N_1}$ such that it is consistent with $w_{N_1}$, $Y_1$ is isometric to $(\overline{\Gamma_{N_1}}, \overline{\rho_1})$ and $(Y_1,\phi_1)$ is isomorphic to $(\overline{\Gamma_{N_1}}, \overline{\Phi_{N_1}})$. 

Since $P_1\subseteq P_2$, we have that all of $N_1$, $H_1$ and $\mathbb{F}(P_1)$ are subgroups of $\mathbb{F}(P_2)$. We will find a coherent extension, ${N_2\trianglelefteqslant  \mathbb{F}(P_2)}$, of $N_1$ of finite index.

Let $\mathcal{G}=G_1*\mathbb{F}(P_2\setminus P_1)$ be the free product of $G_1$ with $\mathbb{F}(P_2\setminus P_1)$. We define a group homomorphism $\psi:\mathbb{F}(P_2)\rightarrow \mathcal{G}$ by letting
\[
\psi(p)=\begin{cases} 
                \Phi_{N_1}(p), &\text{ if } p\in P_1,\\
                p,  \ \ \ \ \ \ &\text{ otherwise}
          \end{cases} \\
 \]
for all $p\in P_2$. Since $H_2$ is a finitely generated subgroup of $\mathbb{F}(P_2)$, $\psi(H_2)$ is a finitely generated subgroup of $\mathcal{G}$. We will find $M\unlhd \mathcal{G}$ of finite index and set $N_2=\psi^{-1}(M)$. To guarantee that $N_2$ is a coherent extension of $N_1$, we need $M$ to satisfy the following corresponding conditions:
\begin{enumerate}
\item[(R1)] For every $p, q, r, s\in P_2\cup\{1\}$ such that $d_{X_2}(p(a_0), q(a_0))\neq d_{X_2}(r(a_0),s(a_0))$, we have $M\cap \psi(p)\psi(H_2)\psi(r)^{-1}\psi(s)\psi(H_2)\psi(q)^{-1}=\emptyset$;
\item[(R2)] For every $p, q\in P_2\cup\{1\}$, if $p(a_0)$ and $q(a_0)$ are defined and $p(a_0)\neq q(a_0)$, we have $M\cap \psi(p)\psi(H_2)\psi(q)^{-1}=\emptyset$;
\item[(R3)] For every $p, q, r_1, s_1, \dots, r_n, s_n\in P_2\cup\{1\}$ such that
$$ d_{X_2}(p(a_0),q(a_0))>\sum_{i=1}^n d_{X_2}(r_i(a_0), s_i(a_0)), $$
we have $$M\cap \psi(p)\psi(H_2)\psi(r_1)^{-1}\psi(s_1)\psi(H_2)\cdots \psi(H_2)\psi(r_n)^{-1}\psi(s_n)\psi(H_2)\psi(q)^{-1}=\emptyset;$$
\item[(S1)] $M\cap G_1=\{1\}$;
\item[(S2)] For every $g, h, k, l\in\mathbb{F}(P_1)$ such that $$d_{Y_1}(\phi_1(g)(a_0), \phi_1(h)(a_0))\neq d_{Y_1}(\phi_1(k)(a_0), \phi_1(l)(a_0)),$$ we have $M\cap \psi(g)\psi(H_2)\psi(k)^{-1}\psi(l)\psi(H_2)\psi(h)^{-1}=\emptyset$;
\item[(S3)] For every $g, h\in \mathbb{F}(P_1)$ and $p,q\in P_2$ with both $p(a_0)$ and $q(a_0)$ defined, if $$d_{Y_1}(\phi_1(g)(a_0),\phi_1(h)(a_0))\neq d_{X_2}(p(a_0),q(a_0)),$$
we have $M\cap \psi(g)\psi(H_2)\psi(p)^{-1}\psi(q)\psi(H_2)\psi(h)^{-1}=\emptyset$.
\end{enumerate}
To see that (S1) implies (D1), note that $N_2\cap\mathbb{F}(P_1)=\psi^{-1}(M)\cap(\psi^{-1}(G_1)\cap \mathbb{F}(P_1))=(\psi^{-1}(M)\cap\psi^{-1}(G_1))\cap \mathbb{F}(P_1)=\psi^{-1}(M\cap G_1)\cap \mathbb{F}(P_1)=\mbox{ker}(\psi)\cap \mathbb{F}(P_1)=N_1$. The other conditions for $M$ obviously imply the corresponding conditions for $N_2$. Note also that each of conditions (R1), (R2) and (R3) is a finite collection of conditions of the form $\gamma M\cap L_1\cdots L_n=\emptyset$ for $\gamma\in \mathcal{G}$ and finitely generated subgroups $L_1,\dots, L_n$ (in fact each $L_i$ is a conjugate of $\psi(H_2)$) with $\gamma\not\in L_1\cdots L_n$. Since $G_1$ is finite, condition (S1) is also a finite collection of conditions of the form $\gamma M\cap \{1\}=\emptyset$ for nonidentity $\gamma\in G_1$. Conditions (S2) and (S3) appear to be about infinitely many elements in $\mathbb{F}(P_1)$. However, since $G_1=\psi(\mathbb{F}(P_1))$ is finite, they all end up being about finitely many elements of $G_1$, and so each of (S2) and (S3) is still a finite collection of conditions of the form $\gamma M\cap L_1\cdots L_n=\emptyset$ for finitely generated subgroups $L_1,\dots, L_n$. We verify that in each case, $\gamma\not\in L_1\cdots L_n$.


Using the $P_2$-type S-map $\phi_{\mathbb{U}}$ from Lemma~\ref{Ucoherent}, we note that for any $g\in \mathbb{F}(P_2)$, if $\psi(g)=1$, then $\phi_{\mathbb{U}}(g)=1$. This follows from the definition of $\psi$ and of $\phi_{\mathbb{U}}$.

For (R1), we need to verify that $1\notin \psi(p)\psi(H_2)\psi(r)^{-1}\psi(s)\psi(H_2)\psi(q)^{-1}$. Toward a contradiction, if $1\in \psi(p)\psi(H_2)\psi(r)^{-1}\psi(s)\psi(H_2)\psi(q)^{-1}$, then there exist $\eta_1,\eta_2\in H_2$ such that $\psi(p\eta_1r^{-1}s\eta_2q^{-1})=1$. Let $\alpha=s\eta_2q^{-1}$ and $\beta=r\eta_1^{-1}p^{-1}$. Then $\psi(\alpha)=\psi(\beta)$ and therefore $\phi_{\mathbb{U}}(\alpha)=\phi_{\mathbb{U}}(\beta)$. Since $\eta_1,\eta_2\in H_2$, $\phi_{\mathbb{U}}(\alpha)(q(a_0))=s(a_0)$ and $\phi_{\mathbb{U}}(\beta)(p(a_0))=r(a_0)$. Now since $\phi_{\mathbb{U}}(\alpha)=\phi_{\mathbb{U}}(\beta)$ is an isometry, we should have $d_{X_2}(p(a_0), q(a_0))= d_{X_2}(r(a_0),s(a_0))$.

For (R2), similar argument shows that if $\alpha=p\eta_1q^{-1}$, then $\phi_{\mathbb{U}}(\alpha)(q(a_0))=p(a_0)$. Now since $\psi(\alpha)=1$, $\phi_{\mathbb{U}}(\alpha)=1$ and therefore $q(a_0)=p(a_0)$.

For (R3), if $$1\in \psi(p)\psi(H_2)\psi(r_1)^{-1}\psi(s_1)\psi(H_2)\cdots \psi(H_2)\psi(r_n)^{-1}\psi(s_n)\psi(H_2)\psi(q)^{-1}$$ then for some $h_1,\dots,h_{n+1}\in H_2$ we have
\[
1=\psi(p)\psi(h_1)\psi(r_1)^{-1}\psi(s_1)\psi(h_2)\cdots \psi(h_n)\psi(r_n)^{-1}\psi(s_n)\psi(h_{n+1})\psi(q)^{-1}.
\]
Consider the sequence 
$$\begin{array}{l}
b_0=\phi_{\mathbb{U}}(p)(a_0)=p(a_0), \\
b_1=\phi_{\mathbb{U}}(ph_1r_1^{-1}s_1)(a_0), \\
b_2=\phi_{\mathbb{U}}(ph_1r_1^{-1}s_1h_2r_2^{-1}s_2)(a_0), \\
\cdots\cdots \\
b_n=\phi_{\mathbb{U}}(ph_1r_1^{-1}s_1h_2\cdots r_n^{-1}s_n)(a_0)= \phi_{\mathbb{U}}(qh_{n+1}^{-1})(a_0)=q(a_0).
\end{array}
$$
We have 
\begin{align*}d_{\mathbb{U}}(b_0, b_1)&=d_{\mathbb{U}}(\phi_{\mathbb{U}}(r_1h_1^{-1}p^{-1})(b_0), \phi_{\mathbb{U}}(r_1h_1^{-1}p^{-1})(b_1))\\
&=d_{\mathbb{U}}(r_1(a_0), s_1(a_0))=d_{X_2}(r_1(a_0), s_1(a_0)),
\end{align*}
and similarly $d_{\mathbb{U}}(b_1, b_2)=d_{X_2}(r_2(a_0), s_2(a_0))$, $\dots$, $d_{\mathbb{U}}(b_{n-1}, b_n)=d_{X_2}(r_n(a_0), s_n(a_0))$. Thus
\[
d_{X_2}(p(a_0),q(a_0))\leq \sum_{i=1}^n d_{X_2}(b_{i-1}, b_i)=\sum_{i=1}^n d_{X_2}(r_i(a_0), s_i(a_0)).
\]

For (S2), we need to verify that if
$$d_{Y_1}(\phi_1(g)(a_0), \phi_1(h)(a_0))\neq d_{Y_1}(\phi_1(k)(a_0), \phi_1(l)(a_0)),$$ 
then $1\notin \psi(g)\psi(H_2)\psi(k)^{-1}\psi(l)\psi(H_2)\psi(h)^{-1}$. Toward a contradiction, assume
$1\in \psi(g)\psi(H_2)\psi(k)^{-1}\psi(l)\psi(H_2)\psi(h)^{-1}$. Then there are $\eta, \eta'\in H_2$ with $$\psi(g)\psi(\eta)\psi(k)^{-1}=\psi(h)\psi(\eta')\psi(l)^{-1}.$$ From the definitions of $H_2$ and of $\psi$, if we apply the left-hand-side element to $\psi(k)(a_0)=\phi_1(k)(a_0)$, the resulting value is $\psi(g)(a_0)=\phi_1(g)(a_0)$. Similarly, if we apply the right-hand-side element to $\psi(l)(a_0)=\phi_1(l)(a_0)$, the resulting value is $\psi(h)(a_0)=\phi_1(h)(a_0)$. Thus, both sides of the equation represent the same partial isometry of $Y_1$ with $\phi_1(k)(a_0)$ and $\phi_1(l)(a_0)$ in its domain and with $\phi_1(g)(a_0)$ and $\phi_1(h)(a_0)$ in its range. We conclude that $d_{Y_1}(\phi_1(g)(a_0), \phi_1(h)(a_0))=d_{Y_1}(\phi_1(k)(a_0), \phi_1(l)(a_0))$, a contradiction.

The argument for (S3) is similar.

Now by Coulbois' theorem (Theorem~\ref{Coulbois}), the group $\mathcal{G}=G_1*\mathbb{F}(P_2\setminus P_1)$ has property RZ. Thus, there exists $M\unlhd\, \mathcal{G}$ of finite index such that all conditions (R1)--(S3) are satisfied. Consequently, $N_2=\psi^{-1}(M)\unlhd \mathbb{F}(P_2)$ is a coherent extension of $N_1$ of finite index.
\end{proof}


\subsection{Extending isometry groups}

In this subsection we apply the algebraic method from the preceding subsection to obtain a construction of S-extensions with prescribed isometry groups. Since we deal with only finite isometry groups, it suffices to consider groups extended by one more generator.

\begin{thm}\label{oneextension}
Let $X_1$ be a finite metric space and $(Y_1,\phi_1)$ be a finite minimal $P_1$-type S-extension of $X_1$. Let $G_1=\phi_1(\mathbb{F}(P_1))$ and $G_2=\langle G_1,k \rangle$ be an overgroup of $G_1$ with one element $k\notin G_1$. Then there exist a finite metric space $X_2\supseteq X_1$, $l\in \mathcal{P}_{X_2}$ and a $P_2$-type S-extension $(Y_2,\phi_2)$ of $X_2$, where $P_2=P_1\cup\{l,l^{-1}\}$, such that $(Y_1,\phi_1)$ and $(Y_2,\phi_2)$ are coherent and $G_2\cong \phi_2(\mathbb{F}(P_2))$.
\end{thm}

\begin{proof} 
By Theorem \ref{lem_onto}, $Y_1$ is isometric to $(\overline{\Gamma_{N_1}}, \overline{\Phi_{N_1}})$ where $N_1=\mbox{\rm ker}(\phi_1)$. Let $X_2=Y_1\cup\{a\}$ be the one point extension of $Y_1$ where $d_{X_2}(a,b)=\mbox{diam}(Y_1)$ for every $b\in Y_1$. Let $P_2=P_1\cup\{l,l^{-1}\}$ where $l=\{(a_0,a)\}$ is the partial isometry that sends $a_0$ to $a$. We define a homomorphism $\phi_2:\mathbb{F}(P_2)\rightarrow G_2$ such that $\phi_2\upharpoonright_{\mathbb{F}(P_1)}=\phi_1$ and $\phi_2(l)=k$. Let $N_2=\mbox{\rm ker}(\phi_2)$. We claim there exists a $G_2$-invariant pseudometric $\rho_2$ on $\Gamma_{N_2}$ which is consistent with $w_{N_2}$,  such that $(Y_2,\phi_2)\cong (\overline{\Gamma_{N_2}}, \overline{\Phi_{N_2}})$ is as desired. Note that because of the definition of $X_2$ and $P_2$, $H_2=H_1$. By Theorem \ref{thmcoherent}, it suffices to show that $N_2$ is a coherent extension of $N_1$. 

\begin{enumerate}
\item[(C1)] For every $p, q, r, s\in P_2\cup\{1\}$ such that $d_{X_2}(p(a_0), q(a_0))\neq d_{X_2}(r(a_0),s(a_0))$, we have $N_2\cap pH_1r^{-1}sH_1q^{-1}=\emptyset$. If $p,q,r,s$ are different from $l$, then since $(\overline{\Gamma_{N_1}}, \overline{\Phi_{N_1}})$ is a $P_1$-type S-extension of $X_1$, by (C0) we have $N_2\cap pH_1r^{-1}sH_1q^{-1}=N_1\cap pH_1r^{-1}sH_1q^{-1}=\emptyset$. If one of $p,q,r,s$ is equal to $l$, then we have $\phi_2(ph_1r^{-1}sh_2q^{-1})=1$ for some $h_1,h_2\in H_1$. Since $l$ appears exactly once in $ph_1r^{-1}sh_2q^{-1}$, this means $\phi_2(l)=k\in G_1$, which is a contradiction. Other cases are obvious.

\item[(C2)] For every $p, q\in P_2\cup\{1\}$, if $p(a_0)$ and $q(a_0)$ are defined and $p(a_0)\neq q(a_0)$, we have $N_2\cap pH_1q^{-1}=\emptyset$. This is similar to (C1).
\item[(C3)] For every $p, q, r_1, s_1, \dots, r_n, s_n\in P_2\cup\{1\}$ such that
$$ d_{X_2}(p(a_0),q(a_0))>\sum_{i=1}^n d_{X_2}(r_i(a_0), s_i(a_0)), $$
we have $N_2\cap pH_1r_1^{-1}s_1H_1\cdots H_1r_n^{-1}s_nH_1q^{-1}=\emptyset$. This is also similar to (C1).
\item[(D1)] $N_1=N_2\,\cap\, \mathbb{F}(P_1)$. If $g\in N_2\,\cap\, \mathbb{F}(P_1)$ then $\phi_1(g)=\phi_2(g)=1$. Therefore, $g\in N_1$.
\item[(D2)] For every $g, h, k, l\in\mathbb{F}(P_1)$ such that $$d_{Y_1}(\phi_1(g)(a_0), \phi_1(h)(a_0))\neq d_{Y_1}(\phi_1(k)(a_0), \phi_1(l)(a_0)),$$ we have $N_2\cap gH_1k^{-1}l H_1h^{-1}=\emptyset$. This is a direct consequence of (D1).
\item[(D3)] For every $g, h\in \mathbb{F}(P_1)$ and $p,q\in P_2$ with both $p(a_0)$ and $q(a_0)$ defined, if $$d_{Y_1}(\phi_1(g)(a_0),\phi_1(h)(a_0))\neq d_{X_2}(p(a_0),q(a_0)),$$
we have $N_2\cap gH_1p^{-1}qH_1h^{-1}=\emptyset$. This is a direct consequence of (D1).
\end{enumerate}
\end{proof}

We remark that it is possible to give a combinatorial proof of Theorem \ref{oneextension}. This is done in \cite{EGLMM} Lemma 5.1, which was in turn motivated by a result of Rosendal (Lemma 16 of \cite{R2}). As in \cite{EGLMM}, Theorem \ref{oneextension} can be used to show that the Hall's universal locally finite group can be embedded as a dense subgroup of the isometry group of the Urysohn space.

\section{Ultraextensive Metric Spaces}

In this section we study ultraextensive metric spaces. 

\begin{defn}\label{uedef} A metric space $U$ is {\it ultraextensive} if 
\begin{enumerate}
\item[(i)] $U$ is ultrahomogeneous, i.e., there is a $\phi$ such that $(U,\phi)$ is an S-extension of $U$;
\item[(ii)] Every finite $X\subseteq U$ has a finite S-extension $(Y,\phi)$ where $Y\subseteq U$;
\item[(iii)] If $X_1\subseteq X_2\subseteq U$ are finite and $(Y_1, \phi_1)$ is a finite minimal S-extension of $X_1$ with $Y_1\subseteq U$, then there is a finite minimal S-extension $(Y_2,\phi_2)$ of $X_2$ such that $Y_2\subseteq U$ and $(Y_1,\phi_1)$ and $(Y_2,\phi_2)$ are coherent.
\end{enumerate}
\end{defn}

Motivated by Hrushovski \cite{H}, Solecki \cite{S} and Vershik \cite{V}, Pestov in \cite{P} introduced a notion of {\it Hrushovski--Solecki--Vershik property}, which is correspondent to the first two clauses of the above definition. He used the notion to study the nonexistence of uniform and coarse embeddings from the universal Urysohn metric space into reflexive Banach spaces. He also gave a proof of Solecki's theorem (Theorem~\ref{thm_S}) using Herwig--Lascar's theorem \cite{HL}. 

Recall that the random graph is the Fra\"{i}ss\'{e} limit of the class of all finite graphs. We equip it with the path metric and turn it into a metric space, which is denoted by $\mathcal{R}$.

\begin{prop}
The Urysohn space $\mathbb{U}$, the rational Urysohn space $\mathbb{Q}\mathbb{U}$ and the random graph $\mathcal{R}$ are ultraextensive.
\end{prop}
\begin{proof} The ultraextensiveness for $\mathbb{U}$ follows directly from its universality and ultrahomogeneity, and from Theorem~\ref{thmcoherent}. 

The space $\mathbb{QU}$ is also ultrahomogeneous and universal for all finite metric spaces with rational distances. From our proof of Theorem~\ref{thm_S} it is clear that if $X$ is a finite metric space with rational distances, then there is a finite S-extension $(Y, \phi)$ of $X$ where the distances of $Y$ are finite sums of the distances in $X$, and therefore also rational. This implies clause (ii) of the definition of ultraextensiveness for $\mathbb{QU}$. The same observation applies to the proof of Theorem~\ref{thmcoherent}. Namely, in every construction of the proof of Theorem \ref{thmcoherent} we used the path (pseudo)metric to define new distances. Thus the distances in $Y_2$ are finite sums of distances in $Y_1\cup X_2$. Therefore, if distances in $X_1,X_2,Y_1$ are rational, then we can find $Y_2$ with rational distances. Together with the ultrahomogeneity and universality of $\mathbb{QU}$, this implies clause (iii) of the definition of ultraextensiveness for $\mathbb{QU}$.

Note that the random graph $\mathcal{R}$ as a metric space has only distances $0,1$ and $2$. In fact, two distinct vertices have distance $1$ if and only if they are connected with an edge. If we endow every finite graph with such a metric, namely, two distinct vertices have distance $1$ if they are connected with an edge, and have distance $2$ otherwise, then $\mathcal{R}$ as a metric space is ultrahomogeneous and universal for this class of finite metric spaces. Then clause (ii) of the definition of ultraextensiveness for $\mathcal{R}$ follows from this universality of $\mathcal{R}$ and from Hrushovski's theorem \cite{H}.  Finally, in Theorem \ref{thmcoherent}, if $X_1,X_2,Y_1$ are finite metric spaces coming from graphs, then they have distances $0,1$ and $2$, and our constructions give that the distances in $Y_2$ are natural numbers. Now if we redefine every distance $\geq 3$ to be $2$ in $Y_2$, then any isometry of $Y_2$ continues to be an isometry in this new metric, and from ultrahomogeneity and universality we again obtain clause (iii) of the definition of ultraextensiveness for $\mathcal{R}$. 
\end{proof}

\begin{thm}\label{cor_Urysohn_1}
Every countable metric space can be extended to a countable ultraextensive metric space. 
\end{thm}

\begin{proof}
Let $X$ be a countable metric space. Write $X$ as an increasing union of finite metric spaces $F_n$ for $n=1,2,\dots$. For $n\geq 1$, inductively define increasing sequences of finite metric spaces $X_n$, $Y_n$ and $Z_n$ as follows. Let $X_1=F_1$ and $(Y_1, \phi_1)$ be a finite minimal S-extension of $X_1=F_1$. We define $Y_1\subseteq Z_1$ such that for every $D\subseteq D'\subseteq Y_1$ and a minimal S-extension of $D$, $(E,\phi)$, where $E\subseteq Y_1$, there exists a minimal S-extension of $D'$, $(E',\phi ')$, where $E'\subseteq Z_1$ and $(E,\phi)$ and $(E',\phi ')$ are coherent. Note that this is possible since there are only finitely many triples $(D,D',E)$ and for any such triple by Theorem \ref{thmcoherent} we can fix a coherent extension $E'$. Finally, to construct $Z_1$, we add $E'\setminus E$ to $Y_1$ for all $E'$ corresponding to the triple $(D,D',E)$ such that the union of the new points ($E'\setminus E$) and $E\subseteq Y_1$ is an isometric copy of $E'$. Then, this new set with the path metric is $Z_1$. Let $X_2$ be the metric space that is obtained by adding $F_2\setminus F_1$ to $Z_1$ such that the union of ($F_2\setminus F_1$) and $F_1$ is isometric to $F_2$ and the distance between points in $F_2\setminus F_1$ and $Z_1\setminus F_1$ comes from the path metric.

 In general, assume finite $Y_{n-1},\subseteq Z_{n-1}\subseteq X_n$ has been defined. Apply Theorem~\ref{thmcoherent} to find $(Y_n, \phi_n)$ a finite minimal S-extension of $X_n\supseteq X_{n-1}$ that is coherent with $(Y_{n-1}, \phi_{n-1})$. We use a similar construction to the construction of $Z_1$ from $Y_1$ to define $Z_n\supset Y_n$. Note that $Z_n$ has the property that every minimal S-extension in $Y_n$ (that is , $D\subseteq E\subseteq Y_n$ where $(E,\phi)$ is a minimal S-extension of $D$) has a coherent minimal S-extension in $Z_n$ for every $D\subseteq D'\subseteq Y_n$. Let $X_{n+1}$ be the metric space that is obtained by adding $F_{n+1}\setminus F_n$ to $Z_n$ such that the union of ($F_{n+1}\setminus F_n$) and $F_n$ is isometric to $F_{n+1}$ and the distance between points in $F_2\setminus F_1$ and $Z_1\setminus F_1$ comes from the path metric.
 
Let $Y$ be the union of the increasing sequence $Y_n$. We verify that $Y$ is ultraextensive. To verify Definition~\ref{uedef} (i), let $p\in\mathcal{P}_Y$. Then there is $n\geq 1$ such that $p\in \mathcal{P}_{X_n}$. Let $n_p$ be the least such $n$. Then for all $m\geq n_p$, $p\subseteq \phi_m(p)\subseteq \phi_{m+1}(p)$ by the coherence of $(Y_m,\phi_m)$ and $(Y_{m+1}, \phi_{m+1})$. Define $\phi(p)=\bigcup_{m\geq n_p}\phi_m(p)$. Then $\phi(p)$ is an isometry of $Y$ that extends $p$. 

For Definition~\ref{uedef} (ii), let $F\subseteq Y$ be finite. Then there is $n$ such that $F\subseteq X_n$, and it follows that $(Y_n, \phi_n\upharpoonright \mathcal{P}_F)$ is an S-extension of $F$.

Finally, for Definition~\ref{uedef} (iii), let $F\subseteq F'\subset Y$ be finite and assume that $(E, \psi)$ is a finite minimal S-extension of $F$ with $E\subseteq Y$. Then, there is a natural number $n$ such that $E\subseteq Y_n$. By the construction of $Z_n$, there exists a minimal S-extension of $F'$, $(E',\phi')$ (corresponding to the triple $(F,F',E)$), such that $E'\subseteq Z_n\subseteq Y$ and that $(E',\phi')$ is coherent with $(E, \phi)$. \qedhere
\end{proof}


\begin{thm}\label{cor_Urysohn_2}
Let $U$ be an ultraextensive metric space and $X\subseteq U$ be a countable subset. Then there exists a countable ultraextensive subset $Y\subseteq U$ with $X\subseteq Y$. 
\end{thm}
\begin{proof}
The proof is similar to that of Theorem \ref{cor_Urysohn_1}. The differences are that in the construction $Y_n$ and $Z_n$ are obtained by applying clauses (ii) and (iii) of the definition of ultraextensive metric space for $U$; and $X_{n+1}=F_{n+1}\cup Z_n$.
\end{proof}

 Pestov [\ref{Pes_B}] showed that $\mbox{Iso}(\mathbb{U})$ contains a countable dense locally finite subgroup. Solecki strengthened this result by showing that $\mbox{Iso}(\mathbb{QU})$ contains a countable dense locally finite subgroup. Rosendal [\ref{Ros_B}] presented a different proof of the result by Solecki. Here we note that such dense locally finite subgroups are present in the isometry group of every separable ultraextensive space.
\begin{thm}\label{locallyfinite}
For every separable ultraextensive metric space $U$, $\mbox{\rm Iso}(U)$ contains a dense locally finite subgroup.
\end{thm}
\begin{proof}
Note that $\mbox{Iso}(U)$ has a countable dense subset $D$. Let $X\subseteq U$ be a countable dense subset with the property that for all $x\in X$ and $\varphi\in D$, $\varphi(x)\in X$. Apply Theorem~\ref{cor_Urysohn_2} to obtain a countable ultraextensive $Y\subseteq U$ with $X\subseteq Y$. Then $\mbox{Iso}(Y)$ is dense in $\mbox{Iso}(U)$. 

It suffices to show that $\mbox{Iso}(Y)$ contains a dense locally finite subgroup. As in the proof of Theorem~\ref{cor_Urysohn_2} we can write $Y$ as an increasing union $\bigcup_n Y_n$. We also have group isomorphic embeddings from each $\mbox{Iso}(Y_n)$ to $\mbox{Iso}(Y_{n+1})$. Let $G=\bigcup_n \mbox{Iso}(Y_n)$. Then it is clear that $G$ is locally finite and $G$ is dense in $\mbox{Iso}(Y)$.
 \end{proof}




\section{Compact Ultrametric Spaces\label{ultrametric}}
In this section we show that every compact ultrametric space can be extended to a compact ultraextensive ultrametric space.  We first study finite ultrametric spaces and show that the notions of homogeneity, ultrahomogeneity, and ultraextensiveness coincide on finite ultrametric spaces.

We will use the following fact about homogeneity for every minimal S-extension.

\begin{lem}\label{homogeneous}
Let $X$ be a metric space and $(Y,\phi)$ be a minimal S-extension of $X$. Then $Y$ is homogeneous.
\end{lem}

\begin{proof}
Let $y_1,y_2\in Y$. Since $(Y,\phi)$ is minimal, there are $g_1,g_2\in \mathbb{F}(\mathcal{P}_X)$ such that $\phi(g_1)(a_0)=y_1$ and $\phi(g_2)(a_0)=y_2$. Hence, $\phi(g_1g_2^{-1})(y_2)=y_1$. Since $\phi(g_1g_2^{-1})$ is an isometry of $Y$, $Y$ is homogeneous.
\end{proof} 
\begin{thm}\label{generalultrametric}
Let $Y$ be a finite ultrametric space. Then the following are equivalent:
\begin{enumerate}
\item[(i)] $Y$ is homogeneous;
\item[(ii)] $Y$ is ultrahomogeneous;
\item[(iii)] $Y$ is ultraextensive.
\end{enumerate}
\end{thm}

\begin{proof}
(i)$\Rightarrow$(ii): Let $D(Y)=\{d(x, y)\,:\, x\neq y\in Y\}$. We prove this by induction on $|D(Y)|$. If $|D(Y)|=1$ then $Y$ is clearly ultrahomogeneous. Suppose $|D(Y)|>1$ and let $r$ be the least element of $D(Y)$. For each $x\in Y$ let $B_r(x)=\{y\in Y\,:\, d(x,y)\leq r\}=\{x\}\cup\{y\in Y\,:\, d(x,y)=r\}$. Then for any $x, y\in Y$, either $B_r(x)=B_r(y)$ or $B_r(x)\cap B_r(y)=\emptyset$. In the latter case, we also have that for any $z_1\in B_r(x)$ and $z_2\in B_r(y)$, $d(z_1, z_2)=d(x,y)$. Let $Y_1=\{B_r(x)\,:\, x\in Y\}$. Then $Y_1$ is a partition of $Y$. For disjoint $B_r(x)$ and $B_r(y)$, we define $d_1(B_r(x), B_r(y))=d(x,y)$. It is easy to check that $(Y_1,d_1)$ is again an ultrametric space, and $D(Y_1)=D(Y)\setminus \{r\}$.  If $\varphi\in \mbox{Iso}(Y)$, then $\varphi$ induces an isometry $\varphi_1$ of $Y_1$, where $\varphi_1(B_r(x))=B_r(\varphi(x))$. Since $Y$ is homogeneous, so is $Y_1$, and by the inductive hypothesis, $Y_1$ is ultrahomogeneous. Now suppose $p: A\to B$ is a partial isometry of $Y$. It induces a partial isometry $p_1: \{B_r(a)\,:\, a\in A\} \to \{B_r(b)\,:\, b\in B\}$ of $Y_1$. Thus there is an isometry $\varphi_1\in \mbox{Iso}(Y_1)$ extending $p_1$. Note that for any $x, y\in Y$, $B_r(x)$ is isometric to $B_r(y)$ by the homogeneity of $Y$, and each $B_r(x)$ is ultrahomogeneous. Now for each $B_r(x)\in Y_1$, we define an isometry from $B_r(x)$ to $\varphi_1(B_r(x))$ as follows. If $B_r(x)\cap A=\emptyset$, then we arbitrarily fix an isometry from $B_r(x)$ to $\varphi_1(B_r(x))$. If $B_r(x)\cap A\neq \emptyset$, then $|B_r(x)\cap A|=|\varphi_1(B_r(x))\cap B|$, and we fix an isometry from $B_r(x)$ to $\varphi_1(B_r(x))$ that sends each $a\in B_r(x)\cap A$ to $p(a)\in \varphi_1(B_r(x))\cap B$. Putting all of these isometries together, we obtain an isometry of $Y$ extending $p$. Thus $Y$ is ultrahomogeneous.

(ii)$\Rightarrow$(iii): We use a similar induction as in the above proof. If $|D(Y)|=1$ then $Y$ is clearly ultraextensive. Assume $|D(Y)|>1$ and let $r$ be the least element of $Y$. Define $Y_1$ similarly as above. Then by the inductive hypothesis $Y_1$ is ultraextensive. For any $x\in Y$, $B_r(x)$ is also ultraextensive. Arbitrarily fix an $x\in Y$ and let $Y_2=B_r(x)$. Consider $Y_1\times Y_2$ and define a metric $d'$ by
$$ d'((B_r(y_1), z_1), (B_r(y_2), z_2))=\max\{d_1(B_r(y_1), B_r(y_2)), d(z_1, z_2)\}. $$
Then $(Y,d)$ is isometric to $(Y_1\times Y_2, d')$. Thus we will view $Y$ as $Y_1\times Y_2$. Enumerate the elements of $Y_1$ by $b_1=B_r(y_1), \dots, b_m=B_r(y_m)$. We show that $Y$ is ultraextensive.

Since $Y$ is finite and ultrahomogeneous, it is enough to show that for every minimal S-extension $(Y_0,\phi_0)$ of $X$ where $X,Y_0\subseteq Y$ there is a group embedding $\pi:\hbox{Iso}(Y_0)\rightarrow \hbox{Iso}(Y)$ such that $\pi(g)\!\upharpoonright\!{Y_0}=g$.

Since $(Y_0,\phi_0)$ is a minimal S-extension, by Lemma \ref{homogeneous}, $Y_0$ is homogeneous and therefore ultrahomogeneous by the previous argument. It follows that the non-empty intersections of $Y_0$ with $b_i=B_r(y_i)$ are isometric. That is, if $Y_0\cap B_r(y_i)\neq \emptyset$ and $Y_0\cap B_r(y_j)\neq \emptyset$, then $Y_0\cap B_r(y_i)$ and $Y_0\cap B_r(y_j)$ are isometric. Arbitrarily fix such a non-empty intersection $Y_{02}$. Let $Y_{01}=\{ B_r(x)\,:\, x\in Y_0\}$. Then $Y_0$ is isometric to $Y_{01}\times Y_{02}$ as a subset of $Y_1\times Y_2$. Now, for every $g\in \hbox{Iso}(Y_0)$, $g$ induces an isometry of $Y_{01}$, which we denote by $\phi_{01}(g)$.  Furthermore, for every $g\in \hbox{Iso}(Y_0)$ and every $1\leq i,j\leq m$ such that $\phi_{01}(g)(b_i)=b_j$, $g$ induces an isometry of $Y_{02}$, which we denote by $\phi(i,j)(g)$. More precisely, if $g(b_i,z_1)=(b_j,z_2)$ for some $1\leq i,j\leq m$ and $z_1,z_2\in Y_{02}$, then $\phi(i,j)(g)(z_1)=z_2$. 
Since $Y_1$ and $Y_2$ are ultraextensive, there are group embeddings $\pi_{01}:\hbox{Iso}(Y_{01})\rightarrow \hbox{Iso}(Y_1)$ and $\pi_{02}:\hbox{Iso}(Y_{02})\rightarrow \hbox{Iso}(Y_2)$ such that $\pi_{01}(g)\!\upharpoonright\!{Y_{01}}=g$ and $\pi_{02}(g)\!\upharpoonright\!{Y_{02}}=g$. Let $\pi:\hbox{Iso}(Y_0)\rightarrow \hbox{Iso}(Y)$ be such that for $b_i\in Y_1$ and $z\in Y_2$ where $\phi_{01}(g)(b_i)=b_j$ we have 
\[
\pi(g)(b_i,z)=(\pi_{01}(\phi_{01}(g))(b_i), \pi_{02}(\phi(i,j)(g))(z)).
\]
Then, $\pi$ is as desired. That is, $\pi$ is a group embedding and if $g(b_i,z)=(b_j,z')$, then $\pi(g)(b_i,z)=(b_j,z')$. Therefore, $Y$ is ultraextensive.

(iii)$\Rightarrow$(i) is obvious.
\end{proof}

In view of Theorem~\ref{generalultrametric} it is easy to construct finite ultrahomogeneous or ultraextensive ultrametric spaces.

\begin{defn}
Let $(\Gamma,w)$ be a connected (undirected) weighted graph. The {\it maximum path metric} on $\Gamma$ is the metric defined by
\begin{align*}
d(x,y)=\inf\{\max\{w(y_i,y_{i+1}) \ : \ i &=1,\dots,n\} \ : \  y_1=x, y_{n+1}=y \text{ and}\\ &(y_i,y_{i+1}) \text{ is an edge in }\Gamma \text{ for all } i=1,\dots,n\}.
\end{align*}
\end{defn}

If $(\Gamma,w)$ is a connected finite weighted graph, then it is easy to see that $\Gamma$ with the maximum path metric is an ultrametric space.

\begin{prop}\label{ultraextensiveultrametric}
Let $X$ be a finite ultrametric space. Then $X$ can be extended to a finite ultraextensive ultrametric space $Y$. Furthermore, there is such $Y$ so that the set of distances in $X$ and $Y$ are the same.
\end{prop} 

\begin{proof} By Theorem~\ref{generalultrametric} it suffices to construct an extension of $X$ that is homogeneous. We use the same notation as in the proof of Theorem~\ref{generalultrametric}. Our proof will be by induction on $|D(X)|$. If $|D(X)|=1$ then $X$ is already homogeneous. Assume $|D(X)|>1$ and let $r$ be the lease element of $D(X)$. Define $X_1=\{B_r(x)\,:\, x\in X\}$ and $d_1$ on $X_1$. Then $|D(X_1)|=|D(X)|-1$. By the inductive hypothesis, $X_1$ can be extended to a homogenous $Y_1$ with the same distances as in $X_1$. Now each $B_r(x)$ is a homogeneous space with every pair of points having distance $r$. Let $N=\max\{|B_r(x)|\,:\, x\in X\}$ and let $x_0\in X$ be such that $|B_r(x_0)|=N$. Then $X_2=B_r(x_0)$ is a homogeneous extension of each of $B_r(x)$. It follows that $Y_1\times X_2$ is a homogeneous ultrametric space extending $X$. 
\end{proof}

\begin{lem}\label{coherentultrametric}
Let $\epsilon>0$. Let $X_1\subseteq X_2$ be finite ultrametric spaces such that $X_1$ is an $\epsilon$-net in $X_2$. Let $(Y_1,\phi_1)$ be a finite minimal S-extension of $X_1$ such that $Y_1$ is an ultrametric space with the same distances as in $X_1$. Then there is a minimal S-extension $(Y_2, \phi_2)$ of $X_2$ such that $Y_2$ is an ultrametric space, $(Y_2,\phi_2)$ is coherent with $(Y_1,\phi_1)$, and $Y_1$ is an $\epsilon$-net in $Y_2$.
\end{lem} 

\begin{proof} By Lemma~\ref{homogeneous} and Theorem~\ref{generalultrametric}, $Y_1$ is homogeneous and therefore ultraextensive. It is enough to construct an S-extension $(Y_2, \phi_2)$ of $X_2$ that satisfies the prescribed conditions, as a minimal S-extension can always be extracted from an S-extension. Since $X_1$ is an $\epsilon$-net in $X_2$, we have that the set of $B_{<\epsilon}(x)=\{y\in X_2\,:\, d_{X_2}(x,y)<\epsilon\}$, when $x$ varies over $X_1$, is a partition of $X_2$. Let $B_{<\epsilon}=X_2/ \sim$ where $\sim$ identifies all points of $X_1$ and the metric on $B_{<\epsilon}$ is the maximum path metric. Then $B_{<\epsilon}$ extends $B_{<\epsilon}(x)$ for every $x\in X_1$ and therefore $X_2$ corresponds to a subset of the product $X_1\times B_{<\epsilon}$. Now $Y_1$ is a homogeneous extension of $X_1$ with the same distances as $X_1$. In particular any distance between distinct points in $Y_1$ is $\geq \epsilon$. We can let $Y_2=Y_1\times B$ where $B$ is a homogeneous extension of $B_{<\epsilon}$ with the same set of distances as $B_{<\epsilon}$. Then $Y_2$ is obviously homogeneous, and therefore ultrahomogeneous. It is clear that $Y_2$ is an ultrametric space, and that for every $y_2\in Y_2$ there is $y_1\in Y_1$ with $d_{Y_2}(y_2,y_1)<\epsilon$. We define a group homomorphism $\phi_2:\mathbb{F}(\mathcal{P}_{X_2})\rightarrow \hbox{Iso}(Y_2)$ such that for every $p\in \mathcal{P}_{X_1}$
\[
\phi_2(p)(y_1,y_2)=(\phi_1(p)(y_1),y_2)
\] 
and for every $p\in \mathcal{P}_{X_2}\setminus \mathcal{P}_{X_1}$ let $\phi_2(p)$ be an isometry of $Y_2$ such that $p\subseteq \phi_2(p)$. Note that since $Y_2$ is ultrahomogeneous, it is possible to find $\phi_2(p)$ as required for every $p\in \mathcal{P}_{X_2}\setminus \mathcal{P}_{X_1}$. It is clear that $(Y_2,\phi_2)$ is coherent with $(Y_1,\phi_1)$.
\end{proof}

\begin{thm}
Every compact ultrametric space can be extended to a compact ultraextensive ultrametric space. In particular, every compact ultrametric space has a compact ultrametric S-extension.
\end{thm}


\begin{proof}
Let $\{X_k\}_{k=1}^\infty$ be an increasing sequence of finite subsets of $X$ such that for each $k,$ $X_k$ is a $\frac{1}{2^k}$-net. Then by Proposition~\ref{ultraextensiveultrametric} and Lemma \ref{coherentultrametric}, there is a sequence of S-extensions $\{(Y_k,\phi_k)\}_{k=1}^\infty$ such that $\{Y_k\}_{k=1}^\infty$ is an increasing sequence of finite ultraextensive ultrametric spaces, $(Y_k,\phi_k)$ is an S-extension of $X_k$, $(Y_{k+1},\phi_{k+1})$ is coherent with $(Y_k,\phi_k)$, and $Y_k$ is a $\frac{1}{2^k}$-net in $Y_{k+1}$. Let $Y$ be the completion of $\bigcup_{k=1}^\infty Y_k$. Then, $Y$ is clearly an ultrametric space; $Y$ is compact since $\bigcup_{k=1}^\infty Y_k$ is totally bounded. In fact, $Y_k$ is a $\frac{1}{2^k}$-net in $Y$. Since each $Y_k$ is ultraextensive, so is $\bigcup_{k=1}^\infty Y_k$. We show that $Y$ is ultraextensive. 

We first show that $Y$ is ultrahomogeneous. For this, let $p: A\to B$ be a partial isometry of $Y$. Let $\frac{1}{2^k}$ be less than the smallest non-zero distance between points of $A$. Since $Y_k$ is a $\frac{1}{2^k}$-net in $Y$, there are $A_k, B_k\subseteq Y_k$ and $p_k: A_k\to B_k$ such that points in $A$ are approximated by points in $A_k$, points in $B$ are approximated by points in $B_k$, and consequently $p_k$ is also a partial isometry. Each $p_k$ can be extended via $\phi_n(p_k)$ for $n> k$ to $\bigcup_{n>k}\phi_n(p_k)$, an isometry of $\bigcup_{k=1}^\infty Y_k$, and then uniquely to an isometry $P_k$ of $Y$. Since $\mbox{Iso}(Y)$ is compact, the collection of $P_k$ has an accumulation point $\varphi$, which is an isometry of $Y$. Since each $P_k$ approxiates $p$ with an error less than $\frac{1}{2^k}$, it follows that $\varphi\supseteq p$. This shows that any partial isometry of $Y$ can be extended to an isometry of $Y$. In particular, it also shows that any partial isometry of $X$ can be extended to an isometry of $Y$, thus there is a suitable $\phi$ such that $(Y, \phi)$ is an S-extension of $X$.

For the remaining properties of ultraextensiveness, it suffices to show that any finite subset of $Y$ can be extended to a finite homogeneous, and therefore ultraextensive, subset of $Y$. For this, let $A\subseteq Y$ be finite and let $\frac{1}{2^k}$ be less than the smallest non-zero distance between points in $A$. Since $Y_k$ is a $\frac{1}{2^k}$-net in $Y$, there is a set $A_k\subseteq Y_k$ such that for each $a\in A$ there is a unique point $a_k\in A_k$ such that $d(a, a_k)<\frac{1}{2^k}$. Consider the set $Z_k=(Y_k\setminus A_k)\cup A$. It is easy to see that the map $\pi: Z_k\to Y_k$ defined by $\pi(a)=a_k$ for $a\in A$ and $\pi(y)=y$ otherwise is an isometry. Thus $Z_k$ is a finite homogenenous subset of $Y$ extending $A$. 
\end{proof}

\section{Open Problems\label{sec:open}}

One general problem is to determine if a certain class of finite metric spaces admit finite S-extensions in the same class. For example, we do not know if the class of finite Euclidean metric spaces has this property.

\begin{ques} Let $X\subseteq \mathbb{R}^n$ be a finite subset. Does $X$ have a finite S-extension $(Y, \phi)$ with $Y\subseteq \mathbb{R}^m$ for some $m\geq n$?
\end{ques}

As stated in Theorem \ref{locallyfinite}, we know that the isometry group of every ultraextensive metric space has a dense locally finite subgroup. It is of interest to know if the isomorphism group of other well-know mathematical objects has the same property. In particular

\begin{ques} Does the homeomorphism group of the Hilbert cube have a dense locally finite subgroup? 
\end{ques}

\begin{ques} Does the linear isometry group of the Gurarij space have a dense locally finite subgroup? 
\end{ques}

\end{document}